\documentclass[12pt]{amsart}
\usepackage[margin = 1.0in]{geometry}

\usepackage{tikz}
\usepackage{tikz-cd}
\usepackage{url, hyperref}
\usepackage{float}
\usepackage[all]{xy}

\tikzset{->-/.style={decoration={
  markings,
  mark=at position .45 with {\arrow{>}}},postaction={decorate}}}

\usepackage{amsmath}
\usepackage{amssymb}
\usepackage{enumitem}
\usepackage{graphicx}
\usepackage{mathdots}
\usepackage{color}
\usepackage{diagbox}
\usepackage{array, makecell}
\usepackage{rotating}
\usepackage{amsthm}

\newtheorem{definition}{Definition}[subsection]
\newtheorem{theorem}[definition]{Theorem}

\newtheorem{proposition}[definition]{Proposition}
\newtheorem{algorithm}[definition]{Algorithm}
\newtheorem{situation}[definition]{Situation}

\newtheorem{lemma}[definition]{Lemma}

\newcommand{\PP}{{\mathbb P}}

\newcommand{\cB}{\mathcal{B}}
\newcommand{\bC}{\mathbb{C}}
\newcommand{\cD}{\mathcal{D}}
\newcommand{\cE}{\mathcal{E}}
\newcommand{\cM}{\mathcal{M}}
\newcommand{\bN}{\mathbb{N}}
\newcommand{\cO}{\mathcal{O}}
\newcommand{\cS}{\mathcal{S}}
\newcommand{\cV}{\mathcal{V}}

\newcommand{\HH}{\mathrm{H}}

\newcommand{\barM}{\overline{\mathcal{M}}}
\newcommand{\bpf}{\text{bpf}}
\newcommand{\Tev}{\mathrm{Tev}}
\newcommand{\logTev}{\mathrm{logTev}}
\newcommand{\Bl}{\mathrm{Bl}}
\newcommand{\vir}{\text{vir}}

\DeclareMathOperator{\ord}{\text{ord}}

\DeclareMathOperator{\Supp}{\text{Supp}}

\newcommand{\mc}[1]{\mathcal{#1}}

\title{Enumerating log rational curves on some toric varieties}

\author{Carl Lian}

\address{Washington University in St. Louis, Department of Mathematics, 1 Brookings Drive
\hfill \newline\texttt{}
\indent St. Louis, MO 63130} \email{{\tt clian@wustl.edu}}

\author{Naufil Sakran}
\address{Tulane University, Department of Mathematics, 6823 St. Charles Ave 
\hfill \newline\texttt{}
\indent New Orleans, LA 70118} \email{{\tt nsakran@tulane.edu}}

\date{\today}

\begin{document}

\maketitle

\begin{abstract}
    The genus 0, fixed-domain log Gromov-Witten invariants of a smooth, projective toric variety $X$ enumerate maps from a general pointed rational curve to a smooth, projective toric variety passing through the maximal number of general points and with prescribed multiplicities along the toric boundary. We determine these invariants completely for the projective bundle $X=\PP_{\PP^r}(\cO^s\oplus\cO(-a))$, proving a conjecture of Cela--Iribar L\'{o}pez. A different conjecture when $X$ is the blow-up of $\PP^r$ at $r$ points is disproven. Whereas the conjectures were predicted using tropical methods, we give direct intersection-theoretic calculations on moduli spaces of ``naive log quasimaps.''
\end{abstract}

%\tableofcontents

\section{Introduction}

\subsection{Counting curves and log curves}

Let $X$ be a smooth, projective variety. Gromov-Witten invariants of $X$ are intersection numbers against the virtual fundamental class $[\barM_{g,n}(X,\beta)]^{\vir}$ on the moduli space of stable maps, and virtually enumerate curves $C$ on $X$ \cite{behrend,bf}. 

Degeneration of the target variety has long played an important role in Gromov-Witten theory \cite{li1,li2}. In the simplest situation, $X$ degenerates into the union of two smooth, projective varieties $X_1\cup X_2$ meeting along a smooth divisor $D$. One hopes to count curves on $X$ in terms of counts of curves of $X_1,X_2$, but must now keep track of the incidences of the curves on $X_j$ with the divisor $D$. This leads naturally to the theory of logarithmic Gromov-Witten invariants of the log scheme $(X_j,D)$ \cite{gs,ac1,ac2}.

Calculations of Gromov-Witten invariants where the (pointed) complex structure of the domain curve is fixed also drew much attention early in the subject. The Vafa-Intriligator formula \cite{bdw,st,mo} determines Gromov-Witten invariants of Grassmannians upon restriction over a fixed point of $\barM_{g,n}$ \cite{mop}. After work of Tevelev \cite{tevelev}, the fixed-domain Gromov-Witten invariants have received renewed interest. The ``virtual Tevelev degrees'' of $X$ are defined by Buch-Pandharipande \cite{bp} to be the virtual degrees of the forgetful morphisms
\begin{equation*}
    \tau:\barM_{g,n}(X,\beta)\to \barM_{g,n}\times X^n,
\end{equation*}
though the term is somewhat ahistorical given the previous work.

A feature of the fixed-domain setting which has motivated much of the more recent work is that, compared to arbitrary Gromov-Witten invariants, the fixed-domain invariants tend to be enumerative much more often \cite{lp,bllrst,celadoan}. Roughly speaking, this means that the virtual invariants compute the \emph{actual} degree of $\tau$, denoted $\Tev^X_{g,n,\beta}$, upon restricting to the open locus of maps $\cM_{g,n}(X,\beta)$ with smooth domain. Concretely, $\Tev^X_{g,n,\beta}$ is the number of maps $f:C\to X$ in class $\beta$ from a fixed, general curve $C$ to $X$ satisfying $n$ incidence conditions $f(p_i)=x_i$ at general points of $C$ and $X$. Moreover, even when the enumerativity fails, the geometric (non-virtual) counts $\Tev^X_{g,n,\beta}$ can sometimes be accessed using different techniques \cite{cps,fl,cl1,cavdaw,cl_hirzebruch,lian_pr,cl2,cl_complete}.

For toric varieties $X$, Cela--Iribar L\'{o}pez \cite{cil} have considered the logarithmic, fixed-domain version of the problem, where the tangency profile of $f$ along the toric boundary of $X$ is constrained. For domain curves of genus 0, they have proven a tropical correspondence theorem \cite[Theorem 5]{cil}, showing that the (geometric or virtual) count of log rational curves on $X$ with the maximal number of incidence conditions imposed at general points matches a purely combinatorial enumeration of tropical curves. In fact, a more general tropical correspondence theorem for log Gromov-Witten invariants of toric varieties in any genus and with arbitrary $\psi$-class insertions (descendants) had been obtained earlier by Mandel-Ruddat \cite{mr}. The fixed-domain setting may be recovered by inserting appropriate $\psi$-classes.

Using the tropical correspondence theorem, the genus 0, logarithmic, fixed-domain invariants for Hirzebruch surfaces are computed \cite[Theorem 7]{cil}, and conjectures for other toric varieties are proposed \cite[Conjectures 14-15]{cil}. See also \cite{zhou,grzz} for further recent work exploiting the connection between log Gromov-Witten theory of projective bundles on toric varieties and tropical curve-counting. The purpose of the present paper is to address the conjectures of Cela--Iribar L\'{o}pez, but in the realm of algebro-geometric intersection theory, rather than tropical geometry. Our approach via moduli spaces of ``naive log quasimaps'' is direct, avoiding the combinatorial analyses typical in tropical enumerations, and leads to explicit formulas. The invariants of interest are reviewed in the next section.

\subsection{Logarithmic Tevelev degrees}\label{sec:logtev}
Let $X/\mathbb{C}$ be a smooth, projective toric variety, and let $D_\rho\subset X$ be the torus-invariant divisors, indexed by maximal rays $\rho\in\Sigma(1)$ in the fan $\Sigma$ of $X$. The union of the $D_\rho$ is referred to throughout as the \emph{boundary} of $X$, and its complement, denoted $X^\circ \subset X$, is referred to as the \emph{interior}.

Let $\beta\in H_2(X)$ be an effective curve class, and let $n\ge 3$ be an integer. Assume that $\int_X\beta\cdot D_{\rho}\ge0$ for all $\rho\in\Sigma(1)$. For each $\rho$, let $\mu_\rho=(\mu_{\rho,v})_{v=1}^{m_\rho}\in \bN^{m_\rho}$ be a vector of positive integers with sum $\int_X\beta\cdot D_{\rho}$. If $\int_X\beta\cdot D_{\rho}=0$, then $\mu_\rho$ is the empty vector and $m_\rho=0$. Write $m=\sum_{\rho\in\Sigma(1)}m_\rho$. We consider maps 
\begin{equation*}
    f:(\PP^1,\{p_i\}_{i=1}^{n},\{\{q_{\rho,v}\}_{\rho\in\Sigma(1)}\}_{v=1}^{m_\rho})\to X
\end{equation*}
where the domain is viewed as a pointed curve in $\cM_{0,n+m}$, so in particular the points $p_i,q_{\rho,v}$ are distinct. Furthermore, we require that
\begin{equation*}
f^{*}D_\rho=\sum_{v=1}^{m_\rho}\mu_{\rho,v}q_{\rho,v}.
\end{equation*}
Thus, the points $q_{\rho,v}$ are exactly those points mapping to the toric boundary of $X$, with multiplicities $\mu_{\rho,v}$. In particular, $f$ maps the $p_i$ to $X^\circ$, and $f_{*}[C]=\beta$. The letter $\Gamma$ is used to denote the combinatorial data of the curve class $\beta$ and the vectors $\mu_\rho$ specifiying the tangency profile of $f$ along the boundary. The additional variable $n$ denotes the number of marked points $p_i$.

\begin{definition}
    Let $\cM_{\Gamma,n}(X)$ be the moduli space of pointed maps $f$ as above. Isomorphisms of maps $f\cong f'$ are commutative diagrams
    \begin{equation*}
\xymatrix{
(\PP^1,\{p_i\}_{i=1}^{n},\{\{q_{\rho,v}\}_{\rho\in\Sigma(1)}\}_{v=1}^{m_\rho}) \ar[rd]_{f} \ar[dd] & \\
 & X \\
(\PP^1,\{p'_i\}_{i=1}^{n},\{\{q'_{\rho,v}\}_{\rho\in\Sigma(1)}\}_{v=1}^{m_\rho})\ar[ru]^{f'} &
}
\end{equation*}
where the vertical arrow is an isomorphism of pointed curves.
\end{definition}

$\cM_{\Gamma,n}(X)$ is a smooth, quasi-projective variety of dimension $(m+n-3)+\dim(X)$. The forgetful morphism
\begin{equation*}
    \tau:\cM_{\Gamma,n}(X)\to \cM_{0,n}\times X^n
\end{equation*}
remembers the pointed curve $(\PP^1,\{p_i\}_{i=1}^{n})\in\cM_{0,n}$ and the collection of evaluated points $\{f(p_i)\}_{i=1}^{n}\in X^n$. If
\begin{equation}\label{eq:dim_constraint}
    n=\frac{m}{\dim(X)}+1,
\end{equation}
then the source and target of $\tau$ have the same dimension.

\begin{definition}
    (cf. \cite[Definition 3]{cil}) Assuming \eqref{eq:dim_constraint}, the (genus 0) \emph{logarithmic} Tevelev degree $\logTev^X_{\Gamma}$ is defined by $\deg(\tau)$.
\end{definition}

For a fixed, general pointed curve $(\PP^1,\{p_i\}_{i=1}^{n})$ and a general choice of points $\{x_i\}_{i=1}^{n}\in X^n$, the integer $\logTev^X_{\Gamma}\ge0$ enumerates maps $f\in \cM_\Gamma(X)$ satisfying $f(p_i)=x_i$. Requiring that $f\in\cM_\Gamma(X)$ constrains $f$ to be incident to each $D_\rho$ with tangency profile $\mu_\rho$.

The logarithmic Tevelev degree $\Tev^X_{\Gamma}$ in \cite{cil} is instead defined by the quotient of $\deg(\tau)$ by the combinatorial factor
\begin{equation*}
    \prod_{\rho\in\Sigma(1)}\prod_{u\ge 1}\#\{u:\mu_{\rho,v}=u\}!.
\end{equation*}
This has the effect of identifying maps which are equal upon permuting marked points $q_{\rho,v}$ with the same multiplicities $u=\mu_{\rho,v}$. This is natural in order for $\Tev^X_\Gamma$ to agree with the usual geometric Tevelev degree $\Tev^X_{0,n,\beta}$ when $\mu_\rho=(1)^{m_\rho}$ for all $\rho\in\Sigma(1)$. However, this combinatorial factor does not seem to arise naturally in the logarithmic calculations, so we omit it in our definition. We use the notation $\logTev^X_{\Gamma}$ to contrast with theirs, but keep the term ``logarithmic Tevelev degree.''

The moduli space of genus 0, log stable maps $\barM_{\Gamma,n}(X)$ is irreducible of expected dimension \cite[Proposition 3.3.5]{ranganathan}, so one can just as well define $\logTev^X_\Gamma$ to be the degree of
\begin{equation*}
    \tau:\barM_{\Gamma,n}(X)\to\barM_{0,n}\times X^n.
\end{equation*}
In this way, $\logTev^X_{\Gamma}$ is given by the integral of a tautological class on $\barM_{\Gamma,n}(X)$. In higher genus, the moduli space of log stable maps no longer has expected dimension, so one should instead integrate against the virtual fundamental class. We work in this paper only in genus 0, where the virtual class is the usual fundamental class.

For calculations, it will be useful to consider log maps to $\PP^1$ where the marked points $p_i$ are not part of the data. Fix the data of $\beta,\mu_\rho$ as above.
\begin{definition}\label{def:maps_no_n}
    Let $\cM_{\Gamma}(X)$ be the moduli space of pointed maps
\begin{equation*}
    f:(\PP^1,\{\{q_{\rho,v}\}_{\rho\in\Sigma(1)}\}_{v=1}^{m_\rho}) \to X
\end{equation*}
with $f_{*}[\PP^1]=\beta$ and the same tangency conditions as before. Isomorphisms of maps $f\cong f'$ are commutative diagrams
    \begin{equation*}
\xymatrix{
(\PP^1,\{\{q_{\rho,v}\}_{\rho\in\Sigma(1)}\}_{v=1}^{m_\rho}) \ar[rd]_{f} \ar[dd]_{\mathrm{id}} & \\
 & X \\
(\PP^1,\{\{q_{\rho,v}\}_{\rho\in\Sigma(1)}\}_{v=1}^{m_\rho})\ar[ru]^{f'} &
}
\end{equation*}
where the vertical arrow is the \emph{identity}.
\end{definition}
That is, $\cM_{\Gamma}(X)$ parametrizes log maps with the same tangency profiles as $\cM_{\Gamma,n}(X)$, except that the domain curve is viewed as a \emph{parametrized} $\PP^1$ (as opposed to a genus 0 curve up to isomorphism), and the points $p_i$ are not part of the data. $\cM_{\Gamma}(X)$ is a smooth, quasi-projective variety of dimension $m+\dim(X)$. We will still be interested in imposing incidence conditions at fixed, general points $p_i\in \PP^1$. After fixing such points, in contrast to $\cM_{\Gamma,n}(X)$, maps $f\in\cM_{\Gamma}(X)$ may send the $p_i$ to the boundary of $X$; equivalently, the points $q_{\rho,v}$ may become equal to the $p_i$.

% To consider the logarithmic Tevelev degree $\Tev^X_{\Gamma}$, we are free also to fix a general domain curve $(\PP^1,\{p_i\}_{i=1}^{n})\in\cM_{0,n}$, and pull back the forgetful morphism $\cM_{\Gamma,n}(X)\to\cM_{0,n}$ over $(\PP^1,\{p_i\}_{i=1}^{n})$. We denote the fiber product simply by $\cM_{\Gamma}(X)$, where the data of the curve class $\beta$ and the tangency profiles $\mu_{\rho}$ are still remembered by $\Gamma$. Then, $\cM_{\Gamma}$ is a smooth, quasi-projective variety of dimension $m+\dim(X)$, and $\Tev^X_{\Gamma}$ is given by the degree of the forgetful morphism
% \begin{equation*}
%     \tau:\cM_{\Gamma}(X)\to X^n.
% \end{equation*}

\subsection{New results}\label{sec:new_results}

In this paper, we consider the genus 0, logarithmic Tevelev degrees of two families of toric varieties, though the techniques should extend in greater generality. The first family we consider is that of the projective bundles
\begin{equation*}
    X_{r,s,a}:=\PP(\cO^s\oplus\cO(-a))\to\PP^r.
\end{equation*}
These are blow-ups of $\PP^{r+s}$ along linear spaces when $a=1$, and Hirzebruch surfaces when $r=s=1$. We completely determine the invariants $\logTev^{X_{r,s,a}}_{\Gamma}$.

\begin{theorem}\label{thm:X_rsa_intro}
Assume that we are in Situation \ref{sit:tev_X_rsa}, so that the logarithmic Tevelev degrees of $X_{r,s,a}$ are defined. 

If all of the inequalities 
    \begin{align*}
        m_j &\le n-1 \text{ for all }j=1,\ldots,r+s+1,\\
        m_{r+2}+\cdots+m_{r+s+1}&\ge (s-1)(n-1),\\
        m_0+m_{r+2}+\cdots+m_{r+s+1}&\le s(n-1)
    \end{align*}
hold, then
    \begin{equation*}
\logTev^{X_{r,s,a}}_{\Gamma}=\left(\prod_{j=0}^{r+s+1}m_j!\right)\left(\prod_{j=0}^{r+s+1}\prod_{v=1}^{m_j}\mu_{j,v}\right)a^{k_0-m_0}\binom{k_0}{m_0},
    \end{equation*}
    where 
    \begin{equation*}
        k_0=s(n-1)-\sum_{j=r+2}^{r+s+1}m_j=\sum_{j=0}^{r+1}m_j-r(n-1).
        \end{equation*}
When $a=0$, the expression $0^0$ is interpreted to equal 1.  
        
Otherwise, we have $\logTev^{X_{r,s,a}}_{\Gamma}=0$.
\end{theorem}

Theorem \ref{thm:X_rsa_intro} appears as Theorem \ref{thm:X_rsa} in the main body of the paper. The torus-invariant divisors $D_j\subset X_{r,s,a}$ are indexed by the integers $j=0,\ldots,r+s+1$, see \S\ref{sec:X_rsa}. The integer $m_0$ appearing in the binomial coefficient of the main formula corresponds in particular to the divisor $D_0$, which in turn is the vanishing locus of the coordinate corresponding to the summand $\cO(-a)$ in $X_{r,s,a}$. The $s=1$ case of Theorem \ref{thm:X_rsa_intro} proves \cite[Conjecture 14]{cil}. See \S\ref{sec:specializations} for other specializations of Theorem \ref{thm:X_rsa_intro}.

Our proof of Theorem \ref{thm:X_rsa_intro} proceeds by an intersection-theoretic calculation on a moduli space $Q_\Gamma(X_{r,s,a})$ of ``naive log quasimaps,'' in which $\cM_\Gamma(X_{r,s,a})$ (Definition \ref{def:maps_no_n}) is a dense open subset, see \S\ref{sec:quasimaps}. These moduli spaces are similar to those considered in \cite{cl2,cl_complete} for blow-ups of $\PP^r$, but incorporate the data of tangency profiles along the toric boundary.

A subtlety is that, unlike on the moduli space of log stable maps $\barM_{\Gamma,n}(X_{r,s,a})$, the natural intersection numbers on $Q_\Gamma(X_{r,s,a})$ are not always enumerative. We prove, however, that in all cases where the enumerativity fails, the logarithmic Tevelev degree is zero. The advantage of the space $Q_\Gamma(X_{r,s,a})$ is that it admits a simple description as a tower of projective bundles over a product of projective lines, so the intersection-theoretic calculation is relatively straightforward.

The space of naive log quasimaps is constructed in \S\ref{sec:moduli}. The enumerativity of tautological intersection numbers on $Q_\Gamma(X_{r,s,a})$ is considered in \S\ref{sec:transversality}, and the calculation of $\logTev^{X_{r,s,a}}_\Gamma$ is completed in \S\ref{sec:logTev_X_rsa}. 

The other $X$ we consider is the blow-up of $\PP^r$ at $r$ (torus-fixed) points, which occupies \S\ref{sec:bl}. While we do not determine the logarithmic Tevelev degrees completely, we show in Theorem \ref{thm:conj15_false} that the prediction of \cite[Conjecture 15]{cil} does not always hold, already when $r=2$. We follow a similar program to the case $X=X_{r,s,a}$, constructing a moduli space $Q_\Gamma(X)$ of naive log quasimaps to $X$. While the value of $\logTev^X_\Gamma$ predicted by \cite[Conjecture 15]{cil} equals a tautological intersection number on our moduli space $Q_\Gamma(X)$, we find that this intersection number can fail to be enumerative even when $\logTev^X_{\Gamma}$ cannot be proven to be zero. Using the excess intersection formula, we compute $\logTev^X_\Gamma$ to be both non-zero and different from the value predicted by \cite[Conjecture 15]{cil} in one example (and the method extends to many more).

\subsection{Conventions}
\begin{itemize}
    \item We work over $\bC$.
    \item The adjectives ``log'' and ``logarithmic'' are used interchangeably.
    \item If $\mc{E}$ is a vector bundle on a scheme $Y$, then $\PP(\mc{E})\to Y$ is the projective bundle of \emph{lines} in the fibers of $\mc{E}$ over $Y$.
\end{itemize}

\subsection{Acknowledgments}

Preliminary discussions related to this project took place at the SLMath (formerly MSRI) summer school on Algebraic Curves in July 2024, and at the workshop ``Tevelev Degrees and related topics'' at UIUC in October 2024. We thank the organizers and other participants of these events for providing an inspiring work environment. We also thank Alessio Cela, Helge Ruddat, and the anonymous referee for comments on drafts of this work. C.L. has been supported by NSF Postdoctoral Fellowship DMS-2001976, an AMS-Simons travel grant, and the Summer Scholars program at Tufts University.

\section{Moduli of log rational curves on $\PP_{\PP^r}(\cO^s\oplus\cO(-a))$}\label{sec:moduli}

In this section, we discuss the moduli spaces needed for the calculation of $\logTev^{X_{r,s,a}}_\Gamma$, where $X_{r,s,a}$ is the toric variety discussed in \S\ref{sec:new_results}.

\subsection{The target variety}\label{sec:X_rsa}
\begin{definition}
    Let $a\geq 0,r\geq 1$ and $s\geq 1$ be integers. Define the $\PP^s$-bundle over $\PP^r$
    \begin{equation*}
        \pi:X_{r,s,a}=\PP_{\PP^r}(\cO^s\oplus\cO(-a))\to\PP^r.
    \end{equation*}
\end{definition}
A point $x\in X_{r,s,a}$ is given by $(r+s+2)$-tuple $(y_0,\ldots,y_{r+s+1})$ of complex numbers, taken up to equivalence under two actions by $\bC^{\times}$:
    \begin{align}\label{eq:X_rsa_scaling}
 \nonumber      (y_0,y_1,\ldots,y_{r+1},y_{r+2},\ldots,y_{r+s+1}) &\sim (\lambda^{-a}y_0,\lambda y_1,\ldots,\lambda y_{r+1},y_{r+2},\ldots,y_{r+s+1}),
        \\
        (y_0,y_1,\ldots,y_{r+1},y_{r+2},\ldots,y_{r+s+1})&\sim (\lambda y_0,y_1,\ldots,y_{r+1},\lambda y_{r+2},\ldots,\lambda y_{r+s+1})\text{ for all } \lambda\in \bC^{\times}.
    \end{align}
Furthermore, the collections of coordinates $\{y_1,\ldots, y_{r+1}\}$ and $\{y_0,y_{r+2},\ldots, y_{r+s+1}\}$ are required not all to vanish. The point $\pi(x)\in\PP^r$ is given by $[y_1:\cdots:y_{r+1}]$, and the fiber of $\pi$ over $[y_1:\cdots:y_{r+1}]$ is given by the projective space with coordinates $y_0,y_{r+2},\ldots,y_{r+s+1}$.

$X_{r,s,a}$ is toric, with action $(\bC^\times)^{r+s}\times X_{r,s,a}\to X_{r,s,a}$ given by
\begin{align*}
    &(t_1,\ldots,t_{r+s})\cdot (y_0,y_1,\ldots,y_{r+1},y_{r+2},\ldots,y_{r+s+1})\\
    =&(y_0,t_1y_1,\ldots,t_ry_r,y_{r+1},t_{r+1}y_{r+2},\ldots,t_{r+s}y_{r+s+1}).
\end{align*}
In this way, the $T$-invariant divisors on $X_{r,s,a}$ are cut out precisely by the $y_j$. Write $D_j=V(y_j)\subset X_{r,s,a}$, so that the set $\Sigma(1)$ of rays in the fan of $X_{r,s,a}$ is identified with the set of indices $\{0,\ldots,r+s+1\}$.

\subsection{Maps to $X_{r,s,a}$}\label{sec:maps}

In this section, we discuss the parametrization of maps from rational curves to $X_{r,s,a}$ with prescribed tangency orders at the boundary. See \cite{cox} for a description of maps to any toric variety. Fix data $\Gamma$ (see \S\ref{sec:logtev}) associated to $X_{r,s,a}$, and let $f:\PP^1\to X_{r,s,a}$ be a map in $\cM_\Gamma(X_{r,s,a})$ (Definition \ref{def:maps_no_n}). The map $f$ fits into a commutative diagram
\begin{equation*}
    \xymatrix{
    \PP^1 \ar[r]^(0.2){f} \ar[rd]_g & X_{r,s,a}=\PP(\cO^s\oplus\cO(-a)) \ar[d]^{\pi} \\
     & \PP^r
    }
    .
\end{equation*}

The map $g:\PP^1\to\PP^r$ is given by a tuple $[g_1:\cdots:g_{r+1}]$, where $g_j\in H^0(\PP^1,\cO(b))$, for some $b\ge 0$. If $g$ is fixed, then $f$ is determined by the data of a surjection
\begin{equation}\label{eq:proj_bundle_sheaf_map}
    \cO^s\oplus\cO(ab) = g^{*}(\cO^s\oplus\cO(-a))^\vee \to \cO(c),
\end{equation}
of sheaves on $\PP^1$, which in turn is given by sections $g_{r+2},\ldots,g_{r+s+1}\in H^0(\PP^1,\cO(c))$ and $g_0\in H^0(\PP^1,\cO(c-ab))$. The sections $g_j$ are taken up to equivalence under the scalings
    \begin{align}\label{eq:section_scaling}
 \nonumber          (g_0,g_1,\ldots,g_{r+1},g_{r+2},\ldots,g_{r+s+1}) &\sim (\lambda^{-a}g_0,\lambda g_1,\ldots,\lambda g_{r+1},g_{r+2},\ldots,g_{r+s+1}),
        \\
     (g_0,g_1,\ldots,g_{r+1},g_{r+2},\ldots,g_{r+s+1})&\sim (\lambda g_0,g_1,\ldots,g_{r+1},\lambda g_{r+2},\ldots,\lambda g_{r+s+1})\text{ for all } \lambda\in \bC^{\times},
    \end{align}
in parallel to \eqref{eq:X_rsa_scaling}. 

In order for the sections $g_j$ possibly to define a map $f\in\cM_\Gamma(X_{r,s,a})$ lying generically in the interior of $X_{r,s,a}$, none of the sections $g_j$ can be identically zero. In particular, we must have $c\ge ab\ge 0$. Thus, we may define, by abuse of notation, the divisors 
\begin{equation*}
    D_j=\text{div}(g_j)\subset \PP^1,
\end{equation*}
which are the pullbacks under $g$ of the corresponding toric divisors $D_j\subset X$. In order for $f\in \cM_\Gamma(X_{r,s,a})$, we need
\begin{equation}\label{Dj_points}
    D_j=\sum_{v=1}^{\mu_j} \mu_{j,v}q_{j,v}
\end{equation}
for some \emph{distinct} points $q_{j,v}\in\PP^1$, where the multiplicities $\mu_{j,v}$ of the points $q_{j,v}\in\Supp(D_j)$ are prescribed by the data $\Gamma$.

We will be interested in the logarithmic Tevelev degrees $\logTev^{X_{r,s,a}}_{\Gamma}$. Let us record the set-up for the problem.

\begin{situation}\label{sit:tev_X_rsa}
    Fix integers $r,s\ge 1$ and $a\ge 0$, and tangency data $\Gamma$ for maps to $X_{r,s,a}$ as above. In particular, we assume that $b\ge 0$ and $c\ge ab\ge 0$. Assume (see \eqref{eq:dim_constraint}) that
    \begin{equation*}
        n=\frac{m}{r+s}+1\ge 3
    \end{equation*}
    is an integer, where $m=\sum_{j=0}^{r+s+1}m_j$ is the total number of distinct intersection points of a rational curve with the toric boundary, as prescribed by $\Gamma$. Fix general points $p_1,\ldots,p_n\in \PP^1$ and $x_1,\ldots,x_n\in X^\circ\subset X_{r,s,a}$, where $X^\circ$ is the interior of $X_{r,s,a}$.
\end{situation}

\subsection{Naive log quasimaps to $X_{r,s,a}$}\label{sec:quasimaps}

In this section, we construct a compactification $Q_\Gamma(X_{r,s,a})$ of $\cM_\Gamma(X_{r,s,a})$ based on the parametrization of the previous section. Define the integers $b,c$ as in \S\ref{sec:maps}, such that $b\ge 0$ and $c\ge ab\ge 0$. 

Let
\begin{equation*}
    \cB := \left(\prod_{v=1}^{m_0} \PP^1_{0,v}\right) \times \cdots \times \left(\prod_{v=1}^{m_{r+s+1}} \PP^1_{r+s+1,v}\right),
\end{equation*}
and let $\nu: \PP^1 \times \cB \longrightarrow \cB$ be the projection. $\cB$ is simply a product of copies of $\PP^1$, parametrizing possible positions of the variable points $q_{j,v}$.

For each $j\in[0,r+s+1]$ and $v\in[1,m_j]$, let $\cD_{j,v}\subset \PP^1\times \cB$ be the pullback of the diagonal from $\PP^1\times \PP_{j,v}$. For each $j$, define the Cartier divisor
\begin{equation*}
    \cD_j=\sum_{v=1}^{m_j}\mu_{j,v}\cD_{j,v}\subset \PP^1\times \cB.
\end{equation*}
\begin{definition}
Define the tower of projective bundles
 \begin{center}
     \begin{tikzcd}
     Q_\Gamma(X_{r,s,a}):=\PP\left(\bigoplus_{j=r+2}^{r+s+1}\nu_{*}(\cO_{\PP^1}(c)(-\mc{D}_j))\oplus \left(\nu_{*}(\cO_{\PP^1}(c-ab)(-\mc{D}_{0}))\otimes \cO_{\PP(\mc{E}_1)}(-a)\right)\right)  
     \arrow[d, "\pi_2"]
     \\
     \PP(\cE_1):=\PP\left(\bigoplus_{j=1}^{r+1} \nu_{*}(\cO_{\PP^1}(b)(-\mc{D}_j))\right)  
     \arrow[d, "\pi_1|_D"]
     \\
     \cB
     \end{tikzcd}
 \end{center}    
\end{definition}

The restrictions of the sheaves $\cO_{\PP^1}(c-ab)(-\cD_{0}),\cO_{\PP^1}(b)(-\cD_{j}),\cO_{\PP^1}(c)(-\cD_j)$ to the fibers of $\nu$ have degree 0, so their push-forwards are line bundles, by Cohomology and Base Change. 

Over a point of $\cB$, which we identify with a collection of divisors $D_j=\sum_{v=1}^{m_j}\mu_{j,v}p_{j,v}\subset\PP^1$, the intermediate projective bundle $\PP(\cE_1)$ parametrizes tuples of sections $g_1,\ldots,g_{r+1}\in H^0(\PP^1,\cO(b))$, not all zero, vanishing along the respective divisors $D_j$. Then, $Q_\Gamma(X)$ parametrizes the additional data of sections $g_0\in H^0(\cO(c-ab))$ and $g_{r+2},\ldots,g_{r+s+1}\in H^0(\cO(c))$, not all zero, again vanishing along the respective $D_j$. The two projective bundles furthermore parametrize these data precisely up to the scaling factors \eqref{eq:section_scaling}.

By construction, $Q_\Gamma(X_{r,s,a})$ is smooth, projective, and irreducible of dimension $m+r+s$. It admits a stratification by open subsets
\begin{equation*}
    \cM_\Gamma(X_{r,s,a})\subset\cM^{\bpf}_\Gamma(X_{r,s,a})\subset Q^{\neq 0}_\Gamma(X_{r,s,a})\subset Q_\Gamma(X_{r,s,a})
\end{equation*}
which we now define. First, the open subset $Q^{\neq 0}_\Gamma(X_{r,s,a})$ is the locus where all sections $g_j$ are not identically zero.

By the parametrization of \S\ref{sec:maps}, we may identify the open subset of $Q^{\neq0}_\Gamma(X_{r,s,a})$ where no two divisors $D_j$ share a common point in their supports with $\cM_\Gamma(X_{r,s,a})$. 

We define $\cM^{\bpf}_\Gamma(X_{r,s,a})$ to be the intermediate open subset of $Q^{\neq0}_\Gamma(X_{r,s,a})$ where the divisors $D_j$ satisfy the following two ``base-point-free'' properties.
\begin{enumerate}
     \item[(BPF1)] The divisors $D_1,\ldots,D_{r+1}$ share no common point $q$ in their support.
     \item[(BPF2)] The divisors $D_0,D_{r+2},\ldots,D_{r+s+1}$ share no common point  $q$ in their support.
\end{enumerate}
The sections $g_j$ underlying a point $f\in \cM^{\bpf}_\Gamma(X_{r,s,a})$ gives rise to a map $f:C\to X_{r,s,a}$ in class $\beta$ in the same way as in \S\ref{sec:maps}. However, the image of such an $f$ may meet the strata of $X_{r,s,a}$ of codimension at least 2, in which case $f\notin \cM_\Gamma(X_{r,s,a})$. We say that $f\in Q_\Gamma(X_{r,s,a})$ is \emph{bpf} if it lies in $\cM^{\bpf}_\Gamma(X_{r,s,a})$.

We refer to $Q_\Gamma(X_{r,s,a})$ as the moduli space of \emph{naive log quasimaps} to $X_{r,s,a}$. Ignoring the marked points $q_{j,v}$, elements $f\in Q_\Gamma(X_{r,s,a})$ are indeed quasimaps in the sense of Ciocan-Fontanine--Kim \cite{cfk}. If $f\in Q_\Gamma(X_{r,s,a})$ is not bpf, then it still defines a \emph{rational} map $\PP^1\dashrightarrow X_{r,s,a}$. After twisting down base-points, this map extends canonically to an actual map, but no longer in class $\beta$. A version of this twisting is made explicit in \S\ref{sec:twisting}.

The adjective ``log'' refers to the fact that the data of the sections $g_j$ underlying $f$ are considered relative to the data of underlying divisors, even when the $g_j$ are identically zero. Our space is ``naive'' in the sense that we allow the points $q_{j,v}$ to collide arbitrarily, rather than requring the data of a stable, logarithmic curve. The actual moduli space of log quasimaps to more general targets is constructed in \cite{shafi}.

We will frequently refer to points of $Q_\Gamma(X_{r,s,a})$ as simply ``quasimaps,'' but the underlying data of the divisors $D_j$ is always present, if implicit.

\subsection{Incidence loci}\label{sec:incidence}

Recall that we are interested in imposing conditions of the form $f(p)=x$ on maps $f\in \cM_\Gamma(X_{r,s,a})$. To do intersection theory, we will need to impose analogous conditions more generally on quasimaps.

\begin{definition}\label{def:V(p,x)}
Let $p\in\PP^1$ and $x\in X^\circ$ be points. Define the \emph{incidence locus} $V(p,x)\subset Q_\Gamma(X_{r,s,a})$ to be the closed subscheme of quasimaps $f$ satisfying the following two sets of conditions:
    \begin{enumerate}
        \item[(1)] $x_j\cdot g_1-x_1\cdot g_j\in H^0(\PP^1,\cO(b))$ vanishes at $p$ for  $j=2,\ldots,r+1$, and
        \item[(2)] $x_j\cdot g_1^ag_0-x_1^ax_0\cdot g_j\in  H^0(\PP^1,\cO(c))$ vanishes at $p$ for $j=r+2,\ldots,r+s+1$. 
    \end{enumerate}    
\end{definition}
%Since $\bar{x}$ is generic, we are safe to assume that all the coordinates $x_i$ of $\bar{x}$ are nonzero. 
The index $1$ plays no essential role. Because $x\in X^\circ$, the equations (1) and (2) imply more generally that
    \begin{enumerate}
        \item[(1)] $x_{j'}\cdot g_j-x_{j}\cdot g_{j'}\in H^0(\PP^1,\cO(b))$ vanishes at $p$ for any $j,j'\in \{1,\ldots,r+1\}$, and
        \item[(2)] $x_j\cdot g_{j'}^ag_0-x_{j'}^ax_0\cdot g_j\in  H^0(\PP^1,\cO(c))$ vanishes at $p$ for any $j\in\{r+2,\ldots,r+s+1\}$ and $j'\in\{1,\ldots,r+1\}$.
    \end{enumerate}
Definition \ref{def:V(p,x)} is given as stated in order for $V(p,x)\subset Q_\Gamma(X_{r,s,a})$ to be cut out by exactly $r+s=\dim(X_{r,s,a})$ equations, and therefore define a cycle class of cohomological degree $r+s$. In particular, $V(p,x)$ has codimension at most $r+s$; in fact, one may easily check that $V(p,x)$ is pure of codimension $r+s$, though this will not be logically necessary for what follows.

For any $p\in\PP^1$ and $x\in X^\circ$, the equations (1) of Definition \ref{def:V(p,x)} are the degeneracy loci of maps of line bundles
\begin{equation*}
    \cO_{\PP(\cE_1)}(-1)\to \bigoplus_{j=1}^{r+1} \nu_{*}(\cO_{\PP^1}(b)(-\cD_j))\to \nu_{*}(\cO_{\PP^1}(b)|_p).
\end{equation*}
on $\PP(\cE_1)$. The first map is the inclusion of the tautological sub-line bundle on $\PP(\cE_1)$. At a point of $\PP(\cE_1)$, its image is the tuple of sections
\begin{equation*}
    (g_1,\ldots,g_{r+1})\in \bigoplus_{j=1}^{r+1} H^0(\PP^1,\cO(b)(-D_j)).
\end{equation*}
The second map, given an index $j\in\{2,\ldots,r+1\}$, returns the restriction of $x_{1}\cdot g_j-x_j\cdot g_{1}$ to $p\in\PP^1$. It is induced by the inclusion $\cO_{\PP^1}(b)(-\cD_j)\subset \cO_{\PP^1}(b)$ (which amounts to viewing the sections $g_j$ instead as sections of $\cO_{\PP^1}(b)$) and the restriction to $p$.

Similarly, the equations (2) are the degeneracy loci of maps of line bundles
\begin{align*}
    \cO_{Q}(-1)&\to \bigoplus_{j=r+2}^{r+s+1}\nu_{*}(\cO_{\PP^1}(c)(-\mc{D}_j))\oplus \left(\nu_{*}(\cO_{\PP^1}(c-ab)(-\cD_{0}))\otimes \cO_{\PP(\cE_1)}(-a)\right)\\
    &\to \nu_{*}(\cO_{\PP^1}(c)|_p),
\end{align*}
where $\cO_Q(-1)$ denotes the universal sub-line bundle on $Q_\Gamma(X_{r,s,a})$, and the first map is again the tautological inclusion. At a point of $Q_\Gamma(X_{r,s,a})$, its image is the tuple of sections
\begin{equation*}
    (g_{r+2},\ldots,g_{r+s+1},g_0)\in \left(\bigoplus_{j=r+2}^{r+s+1} H^0(\PP^1,\cO(c)(-D_j))\right) \oplus H^0(\PP^1,\cO(c-ab)(-D_0)).
\end{equation*}
The second map, given an index $j\in \{r+2,\ldots,r+s+1\}$, returns the restriction of $x_j\cdot g_{1}^ag_0-x_{1}^ax_0\cdot g_j$ to $p\in \PP^1$. In particular, the restriction of second map to the last summand factors through the map
\begin{equation*}
    \nu_{*}(\cO_{\PP^1}(c-ab)(-\cD_{0}))\otimes \cO_{\PP(\cE_1)}(-a) \to \nu_{*}(\cO_{\PP^1}(c-ab)) \otimes (\nu_{*}(\cO_{\PP^1}(b)))^{\otimes a} \to \nu_{*}(\cO_{\PP^1}(c))
\end{equation*}
where the map $\cO_{\PP(\cE_1)}(-a)\to (\nu_{*}(\cO_{\PP^1}(b)))^{\otimes a}$ is induced from the universal sub-line bundle on $\PP(\cE_1)$, and has image $g_{1}^a$ at a given point of $Q_{\Gamma}(X_{r,s,a})$.

The restriction of $V(p,x)$ to the locus of bpf quasimaps $\cM^{\bpf}_\Gamma(X_{r,s,a})$ is precisely the (scheme-theoretic) locus of maps where $f(p)=x$. The closure of this locus in $Q_\Gamma(X_{r,s,a})$ is precisely $V(p,x)$, but this will also not be logically necessary for us.

One can just as easily relativize the construction:
\begin{definition}\label{def:V(p,x)_relative}
Define the closed subscheme
\begin{equation*}
\cV(p,x)\subset Q_\Gamma(X_{r,s,a})\times\PP^1\times X^\circ
\end{equation*}
to be the locus of $(f,p,x)$ satisfying (1) and (2) of Definition \ref{def:V(p,x)}.
\end{definition}
Tautologically, the restriction of $\cV(p,x)$ to any fixed points $p,x$ recovers the absolute subscheme $V(p,x)$.

\section{Transversality}\label{sec:transversality}

The needed moduli spaces in hand, we now consider the calculation of $\logTev^{X_{r,s,a}}_{\Gamma}$. Assume that we are in Situation \ref{sit:tev_X_rsa}. Then, the intersection
\begin{equation*}
    V:=\bigcap_{i=1}^n V(p_i,x_i) \subset Q_\Gamma(X_{r,s,a})
\end{equation*}
is of expected dimension 0. The degree of $V$ (as a 0-cycle) is therefore a candidate for the value of $\logTev^{X_{r,s,a}}_{\Gamma}$. In order to guarantee that $\deg(V)=\logTev^{X_{r,s,a}}_{\Gamma}$ (that is, the intersection number $\deg(V)$ is enumerative), we want $V$ to be reduced of dimension 0, and in addition supported in the locus $\cM_\Gamma(X)\subset Q_\Gamma(X)$. Indeed, if this is the case, then $V$ is identified scheme-theoretically with a general fiber of the morphism $\tau:\cM_{\Gamma,n}(X)\to\cM_{0,n}\times X^n$.

However, this does not always hold. What is actually true is the following.

\begin{proposition}\label{prop:X_rsa_all_transversality}
\quad
    \begin{enumerate}
\item[(i)] (Proposition \ref{prop:transversality_open_locus}) $V\cap \cM_\Gamma(X)$ is reduced of dimension 0.
\item[(ii)] (Proposition \ref{prop:transversality_open_locus}) $V\cap (\cM^{\bpf}_\Gamma(X)-\cM_\Gamma(X))$ is empty.
\item[(iii)] (Proposition \ref{prop:bp_transversality}) $V\cap (Q^{\neq0}_\Gamma(X)-\cM^{\bpf}_\Gamma(X))$ is empty.
\item[(iv)] (Proposition \ref{prop:sections cannot be zero}) If the inequalities
\begin{itemize}
     \item $m_j\le n-1$ for $j=1,\ldots,r+s+1$, and 
     \item $m_{r+2}+\cdots+m_{r+s+1}\ge (s-1)(n-1)$
\end{itemize}
hold, then $V\cap (Q_\Gamma(X)-Q^{\neq0}_\Gamma(X))$ is empty.
    \end{enumerate}
\end{proposition}

The remainder of this section is devoted to proving statements (i)-(iv) above. We will therefore conclude that $\deg(V)=\logTev^{X_{r,s,a}}_{\Gamma}$ under the additional hypotheses of (iv), but not always. The issue will be side-stepped in \S\ref{sec:tev is 0}, where we show that if the hypotheses of (iv) do not hold, then $\logTev^{X_{r,s,a}}_\Gamma=0$.

\subsection{Transversality for maps}

\begin{proposition}\label{prop:transversality_open_locus}
Let $p_i\in\PP^1,x_i\in X^\circ$ be general points, for $i=1,\ldots,n$, and write
\begin{equation*}
    V=\bigcap_{i=1}^n V(p_i,x_i)\subset Q_\Gamma(X_{r,s,a})
\end{equation*}
as above.
\begin{enumerate}
\item[(i)] If $n=\frac{m}{r+s}+1$ \eqref{eq:dim_constraint}, then $V\cap \cM^{\bpf}_\Gamma(X_{r,s,a})$ is reduced of dimension 0. Furthermore, the entirety of the intersection lies in $\cM_\Gamma(X_{r,s,a})$.
\item[(ii)]  If $n>\frac{m}{r+s}+1$, then $V\cap \cM^{\bpf}_\Gamma(X_{r,s,a})$ is empty.
\end{enumerate}
\end{proposition}

\begin{proof}
Consider the intersection
\begin{equation*}
    \cV:=\bigcap_{i=1}^{n}\cV(p_i,x_i)\subset Q_\Gamma(X_{r,s,a})\times(\PP^1)^n\times (X^\circ)^n
\end{equation*}
where $\cV(p_i,x_i)$ denotes the pullback of the subscheme $\cV(p,x)$ (Definition \ref{def:V(p,x)_relative}) from the $i$-th copies of $\PP^1$ and $X^\circ$. Let $\cV^\bpf$ be the restriction of $\cV$ to $\cM^{\bpf}_\Gamma(X_{r,s,a})$. Consider now the composition
\begin{equation*}
\cV^{\bpf} \subset \cM^{\bpf}_\Gamma(X_{r,s,a})\times(\PP^1)^n\times (X^\circ)^n \to \cM^{\bpf}_\Gamma(X_{r,s,a})\times(\PP^1)^n,
\end{equation*}
where the second map is the projection. As we have restricted to the bpf locus, $\cV^{\bpf}$ is simply the subscheme of $(f,\{p_i\},\{x_i\})$ for which $f(p_i)=x_i$ for $i=1,\ldots,n$. Therefore, we may as well consider instead
\begin{equation*}
\overline{\cV^{\bpf}} \subset \cM^{\bpf}_\Gamma(X_{r,s,a})\times(\PP^1)^n\times X_{r,s,a}^n \to \cM^{\bpf}_\Gamma(X_{r,s,a})\times(\PP^1)^n,
\end{equation*}
where $\overline{\cV^{\bpf}}$ is defined by the conditions $f(p_i)=x_i$. In this case, the induced map $\overline{\cV^{\bpf}} \to \cM^{\bpf}_\Gamma(X_{r,s,a})\times(\PP^1)^n$ is clearly an isomorphism: the map $f$ and the points $p_i$ uniquely determine the points $x_i=f(p_i)$. 

In particular, $\overline{\cV^{\bpf}}$, and hence $\cV^{\bpf}$, is smooth of dimension $\dim(\cM_\Gamma(X_{r,s,a}))+n$. The intersection $V\cap \cM^{\bpf}_{\Gamma}(X_{r,s,a})$ in question is simply the generic fiber of the composition
\begin{equation*}
    \cV^{\bpf}\subset \cM^{\bpf}_\Gamma(X_{r,s,a})\times (\PP^1)^n\times (X^\circ)^n \to (\PP^1)^n\times (X^\circ)^n.
\end{equation*}
where the last map is projection.

Similarly, the intersection $V\cap \cM_{\Gamma}(X_{r,s,a})$ in question is simply the generic fiber of the composition
\begin{equation*}
    \cV^{\bpf}|_{\cM_\Gamma(X_{r,s,a})}\subset \cM_\Gamma(X_{r,s,a})\times (\PP^1)^n\times (X^\circ)^n \to (\PP^1)^n\times (X^\circ)^n.
\end{equation*}
where the last map is projection. If $n=\frac{m}{r+s}+1$, then this composition is a map between smooth varieties of the same dimension, so the first part of (i) follows from generic smoothness \cite[Corollary 10.7]{hartshorne}. To see that 
\begin{equation*}
    V\cap (\cM_\Gamma^{\bpf}(X_{r,s,a})-\cM_\Gamma(X_{r,s,a}))=\emptyset,
\end{equation*}
we repeat the argument with $\cM_\Gamma^{\bpf}(X_{r,s,a})$ replaced by the complement of $\cM_\Gamma(X_{r,s,a})$ in $\cM_\Gamma^{\bpf}(X_{r,s,a})$, which has strictly smaller dimension. Then, the restriction of $\cV^{\bpf}$ to the complement of $\cM_\Gamma(X_{r,s,a})$ cannot dominate $(\PP^1)^n\times (X^\circ)^n$ for dimension reasons, so (i) follows.

If $n>\frac{m}{r+s}+1$, then again $\cV^{\bpf}$ cannot dominate $(\PP^1)^n\times (X^\circ)^n$ for dimension reasons, establishing (ii).
\end{proof}

\subsection{Twisting}\label{sec:twisting}

\begin{proposition}\label{prop:bp_transversality}
    Suppose we are in Situation \ref{sit:tev_X_rsa}, and that $f\in Q^{\neq0}_\Gamma(X_{r,s,a})$ lies in $V(p_i,x_i)$ for all $i$. Then, $f$ is bpf. That is, 
\begin{equation*}
    V=\bigcap_{i=1}^n V(p_i,x_i)
\end{equation*}
is empty upon restriction to $Q^{\neq0}_\Gamma(X_{r,s,a})-\cM^{\bpf}_\Gamma(X_{r,s,a})$.
\end{proposition}

The strategy of the proof of Proposition \ref{prop:bp_transversality} is as follows. We will use throughout that $f\in Q^{\neq0}_\Gamma(X_{r,s,a})$ by using that an underlying section $g_j$ vanishes at $p$ if and only if $p\in\Supp(D_j)$. Suppose that $f$ is not bpf. Then, we twist down base-points of $f$ until it defines a map $f'\in \cM_{\Gamma'}(X_{r,s,a})$ for some new tangency data $\Gamma'$. The map $f'$ will retain some incidence constraints, but will have additional constraints on the underlying divisors $D'_j$, namely that they share some common points in their supports. A contradiction will be obtained by counting dimensions.

We now make explicit the twisting procedure. Consider the following two operations on $f\in Q^{\neq0}_\Gamma(X_{r,s,a})$:
\begin{enumerate}
    \item[(T1)] If $q\in \Supp(D_j)$ for all $j=1,\ldots,r+1$ ((BPF1) fails at $q$), then let 
    \begin{equation*}
    y^1_q=\min_{j\in\{1,\ldots,r+1\}}\ord_q(g_j)>0
    \end{equation*}
     be the minimum multiplicity with which $q$ appears in all of these $D_j$.

    Then, replace each $D_j$ ($j=1,\ldots,r+1$) with $D_j-y^1_q\cdot q$, and replace $D_0$ with $D_0+ay^1_q\cdot q$. Replace $b$ with $b-y^1_q$.

    \item[(T2)] If $q\in \Supp(D_j)$ for $j=0$ and all $j=r+2,\ldots,r+s+1$ ((BPF2) fails at $q$), then let 
    \begin{equation*}
    y^2_q=\min_{j\in\{0,r+2,\ldots,r+s+1\}}\ord_q(g_j)>0
    \end{equation*}
     be the minimum multiplicity with which $q$ appears in all of these $D_j$.

    Then, replace each $D_j$ ($j=0,r+2,\ldots,r+s+1$) with $D_j-y^2_q\cdot q$, and replace $c$ with $c-y^2_q$.
\end{enumerate}
One should also choose a re-indexing of the points underlying the new divisors $D_j$, but we suppress this choice. The sections $g_j$ never change. For example, in operation (T1), the sections $g_1,\ldots,g_{r+1}$ initially lie in $H^0(\PP^1,\cO(b)(-D_j))$, which is identified with $H^0(\PP^1,\cO(b-y^1_q)(-(D_j-y^1_q\cdot q)))$ after twisting.

We view the operations (T1) and (T2) as being applied ``dynamically'' on $f$, so $f$ and the underlying divisors $D_j$ are changed, as opposed to new static copies being created. In particular, the tangency data $\Gamma$ changes, but $f$ always remains in the open locus $Q^{\neq 0}_{\Gamma}(X_{r,s,a})$. After (T1) is applied at $q$, $f$ always satisfies property (BPF1) at $q$, but a failure of (BPF2) at $q$ may be introduced. However, operation (T2) always preserves property (BPF1) at $q$, and furthermore property (BPF2) is always satisfied at $q$ after the twist. It is moreover clear that both operations (T1) and (T2) at $q$ preserve properties (BPF1) and (BPF2) at all other points.

Now, consider the following algorithm:
\begin{algorithm}\label{alg:twisting}
Input the quasimap $f\in Q^{\neq0}_\Gamma(X_{r,s,a})$.
\begin{enumerate}
    \item For all points $q\neq p_i$ where applicable, apply (T1) to $f$.
    \item For all points $q\neq p_i$ where applicable, apply (T2) to $f$.
    \item For all points $p_i$ where applicable, apply (T1) to $f$.
    \item For all points $p_i$ where applicable, apply (T2) to $f$.
\end{enumerate}
Denote by $f'$ the value of $f$ after these steps are complete, and by $\Gamma'$ the new tangency data. Output $f'\in \cM^{\bpf}_{\Gamma'}(X_{r,s,a})$.
\end{algorithm}
Because the divisors $D_j$ have finitely many points in their supports, only finitely many operations are applied. By the discussion of the previous paragraph, we have $f'\in\cM^{\bpf}_{\Gamma'}(X_{r,s,a})$, that is, Algorithm \ref{alg:twisting} removes all base-points of $f$. In fact, we will see that $f'$ lies in a subvariety of $\cM^{\bpf}_{\Gamma'}(X_{r,s,a})$ where some of the points underlying the divisors $D_j$ are constrained to be equal.

We now examine more closely the effect of the steps of Algorithm \ref{alg:twisting} on $f$. Specifically, at every step, the twisted $f$ lies in a new moduli space of naive log quasimaps $Q^{\neq 0}_{\Gamma}(X_{r,s,a})_\Delta$, where the data $\Gamma$ is modified, and there are additional ``diagonal'' conditions (denoted by $\Delta$) constraining some of the underlying points $p_i,q_j$ to be equal.

Consider first the application of steps 1 and 2 to a point $q\neq p_i$.
\begin{itemize}
    \item If (T1), but not (T2), is applied at $q$, then write $\alpha^1(q)$ for the number of $g_j$ among $g_1,\ldots,g_{r+1}$ for which $g_j$ vanishes to order exactly $y^1_q$ at $q$. Then, the application of (T1) decreases the value of $m$ underlying $\Gamma$ by either $\alpha^1(q)$ or $\alpha^1(q)-1$, depending on whether or not $D_0$ previously contained $q$ in its support. Moreover, the twisted $f$ has the property that $(r+2)-\alpha^1(q)$ of the twisted divisors $D_0,\ldots,D_{r+1}$ contain $q$ in its support.
    \item If (T2), but not (T1), is applied at $q$, then write $\alpha^2(q)$ for the number of $g_j$ among $g_0,g_{r+2},\ldots,g_{r+s+1}$ for which $g_j$ vanishes to order exactly $y^2_q$ at $q$. Then, the application of (T2) decreases the value of $m$ underlying $\Gamma$ by $\alpha^2(q)$. Moreover, the twisted $f$ has the property that $(s+1)-\alpha^2(q)$ of the twisted divisors $D_0,D_{r+2},\ldots,D_{r+s+1}$ contain $q$ in its support.
    \item If (T1) and (T2) are both applied, then the analyses above may be combined: the value of $m$ decreases by either $\alpha^1(q)+\alpha^2(q)$ or $\alpha^1(q)+\alpha^2(q)-1$, and the twisted $f$ has the property that $(r+s+2)-(\alpha^1(q)+\alpha^2(q))$ of the twisted divisors $D_j$ contain $q$ in its support. Here, $\alpha^2(q)$ is the number of $g_j$ among $g_0,g_{r+2},\ldots,g_{r+s+1}$ for which $g_j$ has the minimal vanishing order at $q$, \emph{after} the twist (T1) is applied.
    %We have accounted here for the possibiility that the condition $\Supp(D_0)\ni q$ has been double-counted.
\end{itemize}
In all three cases, the twisted $f$ lies in the subvariety $Q^{\neq0}_\Gamma(X_{r,s,a})_\Delta\subset Q^{\neq0}_\Gamma(X_{r,s,a})$ where some number (e.g., in the first case, $(r+2)-\alpha^1(q)$) of divisors $D_j$ is constrained to contain $q$. Because the point $q$ itself is unconstrained, $Q^{\neq0}_\Gamma(X_{r,s,a})_\Delta$ has codimension one less than this number of divisors, cut out by the appropriate diagonal on the product $\cB$ of projective lines parametrizing the divisors $D_j$. 

Crucially, in all three cases, the sum of the amount by which $m$ decreases and the number of diagonal conditions is strictly positive. For example, in the first case, we have
\begin{equation*}
    (\alpha^1(q)-1)+[((r+2)-\alpha^1(q))-1]=r>0,
\end{equation*}
and the left hand side is still positive if $(\alpha^1(q)-1)$ is replaced by $\alpha^1(q)$. In the second case, it may happen that $\alpha^2(q)=s+1$, in which case the collection of divisors constrained to contain $q$ after twisting is empty, but the claim still holds in this case.

Moreover, if steps 1 and 2 of Algorithm \ref{alg:twisting} are applied at multiple distinct points $q$, then the diagonal conditions imposed on the $D_j$ are imposed on \emph{pairwise disjoint} subsets of factors of $\cB$. Therefore, after applying steps 1 and 2 at any given point $q\neq p_i$, the dimension of the space of quasimaps $Q^{\neq0}_\Gamma(X_{r,s,a})_\Delta$ in which the twisted $f$ is constrained to lie (where $\Delta$ now keeps track of the \emph{independent} diagonal constraints coming from \emph{all} points $q$ at which twisting has been applied) decreases strictly. Note, finally, that steps 1 and 2 of Algorithm \ref{alg:twisting} have no effect on the incidence conditions imposed at the points $p_i$; it is still true that $f\in V(p_i,x_i)$ for $i=1,\ldots,n$ after twisting.

We now perform a similar analysis for steps 3 and 4 of Algorithm \ref{alg:twisting}. We must now also keep track of the effect on the incidence conditions imposed by $V(p_i,x_i)$.
\begin{itemize}
    \item If only (T2) is applied at $p_i$, then write $\alpha^2(p_i)$ for the number of $g_j$ among $g_0,g_{r+2},\ldots,g_{r+s+1}$ for which $g_j$ vanishes to order exactly $y^2_q$ at $p_i$. Then, the application of (T2) decreases the value of $m$ underlying $\Gamma$ by $\alpha^2(q)$. Moreover, the twisted $f$ has the property that $(s+1)-\alpha^2(q)$ of the twisted divisors $D_0,D_{r+2},\ldots,D_{r+s+1}$ contain $p_i$ in its support.

    The $r$ equations (1) of \ref{def:V(p,x)} still hold for the twisted $f$ at $p_i$, but the equations (2) are no longer all required to hold, because each
    \begin{equation*}
        x_j\cdot g_1^ag_0-x_1^ax_0\cdot g_j \in H^0(\PP^1,\cO(c))
    \end{equation*}
    has been twisted down by a multiple of $p_i$, and the value of $c$ has been decreased accordingly.
    
    \item If (T1) is applied at $p_i$, then (T2) \emph{must} be applied at $p_i$ in step 4. Indeed, after (T1) is applied, the equations (2) of \ref{def:V(p,x)} simply become the condition that $x_1^ax_0\cdot g_j$ vanishes at $p$ for each $j=r+2,\ldots,r+s+1$, but $x_1^ax_0$ is assumed non-zero, so $g_j$ must vanish at $p$, in addition to $g_0$.

    After both (T1) and (T2) are both applied, the value of $m$ decreases by at least $a^1(p_i)+\alpha^2(p_i)-1$ (where $\alpha^1(p_i),\alpha^2(p_i)$ are defined as in the previous analyses), and the twisted $f$ has the property that $(r+s+2)-(\alpha^1(p_i)+\alpha^1(p_i))$ of the twisted divisors $D_j$ contain $p_i$ in its support. In this case, all of the requirements (1) and (2) of \ref{def:V(p,x)} are destroyed for the twisted $f$ at $p_i$.
\end{itemize}
In the first case, the space $Q^{\neq0}_\Gamma(X_{r,s,a})_\Delta$ in which $f$ is constrained to lie decreases in dimension by
\begin{equation*}
\alpha^2(p_i)+((s+1)-\alpha^2(p_i))=s+1
\end{equation*}
after twisting. (Here, we have not yet taken into account the incidence conditions $V(p_i,x_i)$.) The diagonal conditions $\Delta$ now include the constraint that the \emph{fixed} point $p_i$ is contained in the support of some of $D_0,D_{r+2},\ldots,D_{r+s+1}$. It remains true that new diagonal conditions are imposed on distinct factors of $\cB$ than those imposed previously (because the points $p_i$ are distinct from each other and from the points $q$ at which steps 1 and 2 take place), so are independent. In the second case, the dimension $Q^{\neq0}_\Gamma(X_{r,s,a})_\Delta$ decreases by
\begin{equation*}
(\alpha^1(p_i)+\alpha^2(p_i)-1)+[(r+s+2)-(\alpha^1(p_i)+\alpha^2(p_i))]=r+s+1.
\end{equation*}

In both cases, the decrease in dimension in $Q^{\neq0}_\Gamma(X_{r,s,a})_\Delta$ outpaces the number of conditions coming from the incidence loci $V(p_i,x_i)$ that are destroyed by twisting: $s$ in the first case and $r+s$ in the second. This is essentially the content of the proof that follows.

\begin{proof}[Proof of Proposition \ref{prop:bp_transversality}]
     Let $f\in Q^{\neq0}_\Gamma(X_{r,s,a})$ be a point of $V$, and let $f'\in\cM^{\bpf}_{\Gamma'}(X_{r,s,a})_\Delta$ be the result of applying Algorithm \ref{alg:twisting}; the subscript $\Delta$ denotes the collection of diagonal constraints gained in the process, as discussed above.

     For each $i=1,\ldots,n$, there are three cases:
     \begin{enumerate}
         \item Both (T1) and (T2) are applied at $p_i$, in which case there is no constraint on $f'(p_i)$.
         \item Only (T2) is applied at $p_i$, in which case it must be that $f'(p_i)=\pi(x_i)$, where $\pi:X_{r,s,a}\to \PP^r$ is the projection.
         \item Neither (T1) nor (T2) is applied at $p_i$, in which case we must have $f'(p_i)=x_i$.
     \end{enumerate}
    Let $S_1,S_2,S_3$ be the subsets of $\{1,\ldots,n\}$ of indices $i$ in the three respsective cases above, and let $n_1,n_2,n_3$ be their respective cardinalities.
     
    We now allow the $p_i,x_i$ to vary, as in the proof of Proposition \ref{prop:bp_transversality}. For each $i$, define 
    \begin{equation*}
        \cV'_i\subset \cM_{\Gamma'}(X_{r,s,a})_\Delta \times \PP^1 \times X
    \end{equation*}
    to be the locus of $(f',p_i,x_i)$ satisfying the appropriate incidence condition, depending on which of $S_1,S_2,S_3$ in which $i$ lies. Then, the intersection
    \begin{equation*}
        \cV':=\bigcap_{i=1}^{n}\cV'_i\subset \cM_{\Gamma'}(X_{r,s,a})_\Delta \times (\PP^1)^n \times X^n,
    \end{equation*}
    where each $\cV'_i$ is pulled back by the appropriate projection, is pure of dimension
    \begin{equation*}
        \dim(\cM_{\Gamma'}(X_{r,s,a})_\Delta)+(r+s+1)n_1+(s+1)n_2+n_3.
    \end{equation*}

    On the other hand, by the earlier analysis of the change in dimension of $\cM_\Gamma(X_{r,s,a})$ throughout the course of Algorithm \ref{alg:twisting}, we have
    \begin{equation*}
    \dim(\cM^{\bpf}_{\Gamma'}(X_{r,s,a})_\Delta)\le  \dim(\cM_{\Gamma}(X_{r,s,a}))-(r+s+1)n_1-(s+1)n_2,
    \end{equation*}
    with equality if and only if steps 1 and 2 of Algorithm \ref{alg:twisting} were never applied, that is, $f$ has no base-points away from the $p_i$. Here, the term $\dim(\cM_{\Gamma}(X_{r,s,a}))$ refers to the \emph{original} value of $\dim(\cM_{\Gamma}(X_{r,s,a}))$, before any twisting is applied. By the assumption of Situation \ref{sit:tev_X_rsa}, we have $\dim(\cM_{\Gamma}(X_{r,s,a}))=(r+s)n$. 

    Therefore, we conclude that
    \begin{align*}
        \dim(\cV')&=\dim(\cM_{\Gamma'}(X_{r,s,a})_\Delta)+(r+s+1)n_1+(s+1)n_2+n_3\\
        &\le [\dim(\cM_{\Gamma}(X_{r,s,a}))-(r+s+1)n_1-(s+1)n_2]+[(r+s+1)n_1+(s+1)n_2+n_3]\\
        &=(r+s)n+n_3\\
        &\le (r+s+1)n=\dim((\PP^1)^n \times X^n)
    \end{align*}
    The first inequality is an equality if and only if no twisting is applied at points $q\neq p_i$, and the second inequality is an equality if and only if no twisting is applied at the $p_i$. 
    
    % However, because the given $p_i,x_i$ are chosen generally, $\cV$ must dominate $(\PP^1)^n \times X^n$, so $\dim(\cV')\ge \dim((\PP^1)^n \times X^n)$. Therefore, equality must hold everywhere. This means that no twisting was applied at all, which is to say that $f\in\cM^{\bpf}_\Gamma(X_{r,s,a})$.

    However, because the given $p_i,x_i$ are chosen generally, $\cV'$ must dominate $(\PP^1)^n \times X^n$. Indeed, the assumption is that there is at least one $f'\in \cM_{\Gamma'}(X_{r,s,a})_\Delta$ satisfying the needed incidence conditions at a general collection of points $p_i,x_i$. In particular, we need $\dim(\cV')\ge \dim((\PP^1)^n \times X^n)$. Therefore, equality must hold everywhere. This means that no twisting was applied at all, which is to say that $f\in\cM^{\bpf}_\Gamma(X_{r,s,a})$.
\end{proof}

\subsection{Vanishing sections}
          
\begin{proposition}\label{prop:sections cannot be zero}
 Suppose we are in Situation \ref{sit:tev_X_rsa}, and that $f\in Q_\Gamma(X_{r,s,a})$ lies in $V(p_i,x_i)$ for all $i$. Assume further that:
 \begin{itemize}
     \item $m_j\le n-1$ for $j=1,\ldots,r+s+1$, and 
     \item $m_{r+2}+\cdots+m_{r+s+1}\ge (s-1)(n-1)$.
 \end{itemize}
 Then, $f\in Q^{\neq0}_{\Gamma}(X_{r,s,a})$. That is,
\begin{equation*}
    V=\bigcap_{i=1}^n V(p_i,x_i)
\end{equation*}
is empty upon restriction to $Q_\Gamma(X_{r,s,a})-Q^{\neq0}_\Gamma(X_{r,s,a})$.
\end{proposition}

\begin{proof}
Assume for sake of contradiction that some of the sections $g_j$ underlying $f$ are identically zero. 

Suppose first that $g_j=0$ identically for some $j\in\{1,\ldots,r+1\}$. Then, by the equations (1) of Definition \ref{def:V(p,x)}, it must be the case that all of $g_1,\ldots,g_{r+1}\in H^0(\PP^1,\cO(b))$ vanish at all of the points $p_i$. However, because $m_{j}\le n-1$ for $j=1,\ldots,r+1$, this is only possible if $g_1,\ldots,g_{r+1}$ are identically zero. This is a contradiction, as $g_1\ldots,g_{r+1}$ cannot all be identically zero by the construction of $Q_\Gamma(X_{r,s,a})$.

Next, suppose that $g_0=0$ identically. Then, by the equations (2) of Definition \ref{def:V(p,x)}, all of $g_{r+2},\ldots,g_{r+s+1}$ must vanish at all of the points $p_i$. If $m_j\le n-1$ for $j=r+2,\ldots,r+s+1$, then  $g_0,g_{r+2},\ldots,g_{r+s+1}$ must all be zero, again a contradiction.

Finally, if any of $g_{r+2},\ldots,g_{r+s+1}$ is identically zero, then all of them must be. Furthermore, we need $g_0g_j^a$ to vanish at all of $p_1,\ldots,p_n$, for $j=1,\ldots,r+1$. Thus, for each $p_i$, either $g_0$ vanishes at $p_i$, or all of $g_1,\ldots,g_{r+1}$ vanish at $p_i$. Furthermore, we have seen above that none of $g_0,\ldots,g_{r+1}$ may be identically zero.

Without loss of generality, suppose that $g_0$ vanishes at $p_1,\ldots,p_\ell$, but not $p_{\ell+1},\ldots,p_n$. Then, we must have $m_0\ge \ell$. Furthermore, the tuple $g=[g_1:\cdots:g_{r+1}]$ has the property that all $r+1$ sections vanish at the remaining points $p_{\ell+1},\ldots,p_n$, and satisfies $r$ additional vanishing conditions $(x_j\cdot g_1-x_1\cdot g_j)(p_i)=0$ for $i=1,\ldots,\ell$. The tuple of non-zero sections $g$ may be viewed as a naive log quasimap to $\PP^r$; repeating and suitably modifying the arguments as in Propositions \ref{prop:bp_transversality} and \ref{prop:transversality_open_locus} shows that we need
\begin{equation*}
    (m_1+\cdots+m_{r+1})+r\ge r\ell+(r+1)(n-\ell).
\end{equation*}
The left hand side is the dimension of $\cM_\Gamma(\PP^r)$, where the tangency data $\Gamma$ is obtained by restricting the tangency data in $Q_\Gamma(X_{r,s,a})$. The right hand side is the number of vanishing conditions imposed by the points $p_i$.

Combining with the fact that $m_0\ge \ell$ and $m=(n-1)(r+s)$ gives
\begin{equation*}
    m_{r+2}+\cdots+m_{r+s+1}\le s(n-1)-n,
\end{equation*}
which contradicts the stated assumption. 
\end{proof}

\section{Logarithmic Tevelev degrees of $X_{r,s,a}$}\label{sec:logTev_X_rsa}

We are now ready to carry out the calculation of $\logTev^{X_{r,s,a}}_{\Gamma}$. Assume that we are in Situation \ref{sit:tev_X_rsa}, and consider again the intersection
\begin{equation*}
    V=\bigcap_{i=1}^{n}V(p_i,x_i)\subset Q_\Gamma(X_{r,s,a}).
\end{equation*}

By Proposition \ref{prop:X_rsa_all_transversality}, we have
\begin{equation}\label{eq:tev_equals_integral}
    \logTev^{X_{r,s,a}}_{\Gamma}=\int_{Q_\Gamma(X_{r,s,a})}\prod_{i=1}^{n}[V(p_i,x_i)].
\end{equation}
under the additional assumptions of Proposition \ref{prop:sections cannot be zero} (Proposition \ref{prop:X_rsa_all_transversality}(iv)). The integral on the right hand side can be computed directly; this is done in \S\ref{sec:computation}. We first deal with the case in which these assumptions do not hold.

\subsection{Vanishing}\label{sec:tev is 0}

\begin{proposition}\label{prop:tev_nonzero_constraints}
 Suppose we are in Situation \ref{sit:tev_X_rsa}, and that there exists at least one map $f\in \cM_\Gamma(X_{r,s,a})$ with $f(p_i)=x_i$ for all $i=1,\ldots,n$. Then, the hypotheses of Proposition \ref{prop:sections cannot be zero} hold, that is, we must have:
 \begin{itemize}
     \item $m_j\le n-1$ for $j=1,\ldots,r+s+1$, and 
     \item $m_{r+2}+\cdots+m_{r+s+1}\ge (s-1)(n-1)$.
 \end{itemize}
 In particular, if these inequalities do not all hold, then $\logTev^{X_{r,s,a}}_\Gamma=0$.
\end{proposition}

\begin{proof}
Let $g_0,\cdots,g_{r+s+1}$ be the sections underlying $f$.

        We show first that $m_\ell\le n-1$ for $\ell\in\{1,\ldots,r+1\}$. Suppose first that $r>1$. Then, dropping the section $g_\ell$, the remaining sections $g_0,\ldots,\widehat{g_\ell}\ldots,g_{r+s+1}$ define a point $f'\in \cM_{\Gamma'}(X_{r-1,s,a})$, for the restricted tangency data $\Gamma'$. The divisors $D_j$ underlying the remaining $g_j$ still have disjoint supports, and the remaining $g_j$ are still non-zero. Furthermore, we have $f'(p_i)=x'_i$, where $x'_i$ is the projection $X_{r,s,a}\dashrightarrow X_{r-1,s,a}$ defined by forgetting the coordinate indexed by $j$. Therefore, by Proposition \ref{prop:transversality_open_locus}(ii), we must have
        \begin{equation*}
            n\le \frac{m-m_\ell}{r+s-1}+1=\frac{(r+s)(n-1)-m_\ell}{r+s-1}+1,
        \end{equation*}
        which is equivalent to $m_\ell\le n-1$. If instead $r=1$, then for $\ell=1,2$, the $(s+1)$-tuple of sections $[g_0g_\ell^a:g_3:\cdots:g_{s+2}]$ defines a (log) map to $\PP^1\to \PP^s$ satisfying $n$ incidence conditions at general points. We find similarly, repeating the proof of Proposition \ref{prop:transversality_open_locus} for $X=\PP^s$, that $m_\ell\le n-1$.

        Next, consider $\ell\in\{r+2,\ldots,r+s+1\}$. If $s>1$, then the sections $g_0,\ldots,\widehat{g_\ell}\ldots,g_{r+s+1}$ define a point of $f'\in \cM_{\Gamma'}(X_{r,s-1,a})$ satisfying $n$ general incidence conditions. Arguing as above gives $m_\ell\le n-1$ once more. If $s=1$, then we consider instead the map $g:\PP^1\to\PP^r$ defined by $[g_1:\cdots:g_{r+1}]$. We obtain in this case the stronger inequality $m_0+m_{r+2}\le n-1$.

        Finally, we prove the inequality $m_{r+2}+\cdots+m_{r+s+1}\ge (s-1)(n-1)$. There is nothing to check if $s=1$, so we assume that $s>1$, and consider the map $h:\PP^1\to\PP^{s-1}$ defined by $[g_{r+2}:\cdots:g_{r+s+1}]$, which satisfies $n$ general incidence conditions. Therefore, we must have
        \begin{equation*}
            (s-1)n\le \dim\cM_{\Gamma'}(\PP^{s-1})=(m_{r+2}+\cdots+m_{r+s+1})+(s-1),
        \end{equation*}
        which is equivalent to the needed inequality.
\end{proof}

\subsection{Integral computation}\label{sec:computation}

It is left to compute the integral 
\begin{equation*}
    \int_{Q_\Gamma(X_{r,s,a})}\prod_{i=1}^{n}[V(p_i,x_i)]
\end{equation*}
appearing in \eqref{eq:tev_equals_integral}. The integrand is a 0-cycle if $n=\frac{m}{r+s}+1$ \eqref{eq:dim_constraint}.

Recall the set-up of \S\ref{sec:quasimaps}. The product
\begin{equation*}
    \cB := \left(\prod_{v=1}^{m_0} \PP^1_{0,v}\right) \times \cdots \times \left(\prod_{v=1}^{m_{r+s+1}} \PP^1_{r+s+1,v}\right),
\end{equation*}
parametrizes choices of the divisors $D_j$, and $\nu: \PP^1 \times \cB \longrightarrow \cB$ is the projection. The universal divisors are denoted $\cD_j\subset \PP^1\times \cB$. We have the rank $r+1$ vector bundle
\begin{equation*}
\cE_1:=\bigoplus_{j=1}^{r+1} \nu_{*}(\cO_{\PP^1}(b)(-\cD_j))
\end{equation*}
on $\cB$, giving rise to the projective bundle $\pi_1:\PP(\cE_1)\to\cB$. We have the rank $s+1$ vector bundle
\begin{equation*}
\cE_2:=\bigoplus_{j=r+2}^{r+s+1}\nu_{*}(\cO_{\PP^1}(c)(-\mc{D}_j))\oplus \left(\nu_{*}(\cO_{\PP^1}(c-ab)(-\cD_{0}))\otimes \cO_{\PP(\cE_1)}(-a)\right)
\end{equation*}
on $\PP(\cE_1)$, where $\cO_{\PP(\cE_1)}(-1)$ is the universal sub-line bundle, giving rise to the projective bundle $\pi_2:Q_\Gamma(X_{r,s,a})=\PP(\cE_2)\to\PP(\cE_1)$.

Write $\zeta_1,\zeta_2$ for the relative hyperplane classes on $\PP(\cE_1)$ and $\PP(\cE_2)$, respectively. Write $\HH_{j,v}$ for the hyperplane class on the factor $\PP^1_{j,v}$ of $\cB$, and 
\begin{equation*}
    \HH_j=\sum_{v=1}^{m_j}\mu_{j,v}\HH_{j,v}.
\end{equation*}
A standard calculation shows that
\begin{align*}
    c_1(\nu_{*}(\cO_{\PP^1}(c-ab)(-\cD_0)))&=-\HH_0,\\
    c_1(\nu_{*}(\cO_{\PP^1}(b)(-\cD_j)))&=-\HH_j\text{ for }j=1,\ldots,r+1,\\
    c_1(\nu_{*}(\cO_{\PP^1}(c)(-\cD_j)))&=-\HH_j\text{ for }j=r+2,\ldots,r+s+1.\\
\end{align*}

For any $p\in\PP^1$ and $x\in X^\circ$, recall from \S\ref{sec:incidence} that $V(p,x)$ is cut out by the vanishing of $r$ maps of line bundles
\begin{equation*}
    \cO_{\PP(\cE_1)}(-1) \to \nu_{*}(\cO_{\PP^1}(b)|_p)
\end{equation*}
and $s$ maps of line bundles 
\begin{equation*}
    \cO_{Q}(-1) \to \nu_{*}(\cO_{\PP^1}(c)|_p).
\end{equation*}
Therefore, we have
\begin{align*}
[V(p,x)]&=(c_1(\nu_{*}(\cO_{\PP^1}(b)|_p)))-c_1(\cO_{\PP(\cE_1)}(-1)))^r(c_1(\nu_{*}(\cO_{\PP^1}(c)|_p)))-c_1(\cO_{\PP(\cE_2)}(-1)))^s\\
&=\zeta_1^r\zeta_2^s,
\end{align*}
because the line bundles $\nu_{*}(\cO_{\PP^1}(b)|_p)$ and $\nu_{*}(\cO_{\PP^1}(c)|_p)$ are trivial, and 
\begin{align*}
-c_1(\cO_{\PP(\cE_1)}(-1))&=c_1(\cO_{\PP(\cE_1)}(1))=\zeta_1,\\
-c_1(\cO_{\PP(\cE_2)}(-1))&=c_1(\cO_{\PP(\cE_1)}(1))=\zeta_2
\end{align*}
by definition.
    
    Now, 
    \begin{align*}
    \int_{\PP(\mc{E}_2)}\prod_{i=1}^{n}[V(p_i,x_i)]^n &= \int_{\PP(\mc{E}_2)}\pi_2^{*}(\zeta_1^{rn})\cdot\zeta_2^{sn}
    \\
    &=\int_{\PP(\mc{E}_1)}\zeta_1^{rn}\cdot(\pi_2)_{*}(\zeta_2^{sn}) \\
    &=\int_{\PP(\mc{E}_1)}\zeta_1^{rn}\cdot \cS_{s(n-1)}(\mc{E}_2) 
    \end{align*}
    where $\cS$ denotes Segre class. We have

    \begin{equation*}
    \mc{S}(\mc{E}_2)=\frac{1}{c(\mc{E}_2)}=\frac{1}{(1-\HH_0-a\zeta_1)(1-\HH_{r+2})\cdots(1-\HH_{r+s+1})}
    \end{equation*}

    so

    \begin{align*}
   \int_{\PP(\mc{E}_2)}\prod_{i=1}^{n}[V(p_i,x_i)]^n  
    &=\int_{\PP(\mc{E}_1)}\zeta_1^{rn}\cdot \left(\sum_{k_0+k_{r+2}+\cdots +k_{r+s+1}=s(n-1)}(\HH_0+a\zeta_1)^{k_0}\HH_{r+2}^{k_{r+2}}\cdots \HH_{r+s+1}^{k_{r+s+1}}\right)
    \\
    &=\int_{\PP(\mc{E}_1)} \sum_{k_0+k_{r+2}+\cdots +k_{r+s+1}=s(n-1)}\sum_{\ell=0}^{k_0}\left(\binom{k_0}{\ell}a^{k_0-\ell}\zeta_1^{rn+k_0-\ell} \HH_0^\ell\right)\HH_{r+2}^{k_{r+2}}\cdots \HH_{r+s+1}^{k_{r+s+1}}
    \\
    &=\int_{\cB} \sum_{k_0+k_{r+2}+\cdots +k_{r+s+1}=s(n-1)}\sum_{\ell=0}^{k_0}\left(\binom{k_0}{\ell}a^{k_0-\ell}(\pi_1)_{*}(\zeta_1^{rn+k_0-\ell}) \HH_0^\ell\right)\HH_{r+2}^{k_{r+2}}\cdots \HH_{r+s+1}^{k_{r+s+1}}
    \\
    \end{align*}
    The $k_\ell$ are always assumed to be non-negative integers. When $a=0$, we take the convention
    \begin{equation*}
    0^k=\begin{cases}
            1,&k=0\\
             0,&k\neq 0
         \end{cases}.
    \end{equation*}
    
    Now, we have
    \begin{equation*}
        \mc{S}(\mc{E}_1)=\frac{1}{c(\mc{E}_1)}=\frac{1}{(1-\HH_1)\cdots(1-\HH_{r+1})},
    \end{equation*}

    so replacing $(\pi_{1})_{*}(\zeta_1^{rn+k_0-\ell})$ with $\mc{S}_{r(n-1)+k_0-\ell}(\mc{E}_1)$ gives
    \begin{align*}
    \int_{\PP(\mc{E}_2)}\prod_{i=1}^{n}[V(p_i,x_i)]^n
    &=
    \int_{\cB} \sum_{k_0+k_{r+2}+\cdots +k_{r+s+1}=s(n-1)}\sum_{\ell=0}^{k_0}\binom{k_0}{\ell}a^{k_0-\ell}\mc{S}_{r(n-1)+k_0-\ell}(\mc{E}_1) \HH_0^\ell \HH_{r+2}^{k_{r+2}}\cdots \HH_{r+s+1}^{k_{r+s+1}}
    \\
    &= 
    \int_{\cB} \sum_{k_0+k_{r+2}+\cdots +k_{r+s+1}=s(n-1)}\sum_{\ell=0}^{k_0}\binom{k_0}{\ell}a^{k_0-\ell}\sum_{k_1+\cdots +k_{r+1}=r(n-1)+k_0-\ell} \HH_0^\ell\prod_{j=1}^{r+s+1}\HH_j^{k_j}.
    \end{align*}
    The value of the integral is the coefficient of $\prod_{j=0}^{r+s+1}\prod_{v=1}^{m_j} \HH_{j,v}$ of the integrand, where we recall that
    \begin{equation*}
        \HH_j=\sum_{v=1}^{m_j}H_{j,v}.
    \end{equation*}
    Thus, in order to obtain a non-trivial contribution, one must take $k_j=m_j$ for $j=1,\ldots,r+s+1$.

    Therefore,
    \begin{equation*}
    \sum_{j=1}^{r+1} m_i=r(n-1)+k_0-\ell\implies k_0=\sum_{j=1}^{r+1} m_i-r(n-1)+\ell,
    \end{equation*}
    and
    \begin{align*}
        k_0+\sum_{i=r+2}^{r+s+1}m_j=s(n-1)
        &\implies \sum_{j=1}^{r+s+1}m_j-r(n-1)+\ell=s(n-1)\\
        &\implies \ell=m_0,
    \end{align*}
    because $n=\frac{m}{r+s}+1$. Note that this is only valid if $\ell=m_0\le k_0$, because $\ell$ ranges between $0$ and $k_0$ in the integrand. Thus, if $m_0>k_0$, then the integral is zero.

    Substituting the above values of $k_j$ and $\ell$ gives now that, if $k_0\ge m_0$, then
    \begin{equation}\label{eq:X_rsa_answer}
    \int_{\PP(\mc{E}_2)}\prod_{i=1}^{n}[V(p_i,x_i)]^n=\left(\prod_{j=0}^{r+s+1}m_j!\right)\left(\prod_{j=0}^{r+s+1}\prod_{v=1}^{m_j}\mu_{j,v}\right)a^{k_0-m_0}\binom{k_0}{m_0},
    \end{equation}
    where
    \begin{equation}\label{eq:k_0}
        k_0=s(n-1)-\sum_{j=r+2}^{r+s+1}m_j=\sum_{j=0}^{r+1}m_j-r(n-1).
    \end{equation}

\begin{theorem}\label{thm:X_rsa}
    When defined (Situation \ref{sit:tev_X_rsa}), we have
    \begin{equation*}
\logTev^{X_{r,s,a}}_{\Gamma}=\left(\prod_{j=0}^{r+s+1}m_j!\right)\left(\prod_{j=0}^{r+s+1}\prod_{v=1}^{m_j}\mu_{j,v}\right)a^{k_0-m_0}\binom{k_0}{m_0},
    \end{equation*}
    where $k_0$ is defined in \eqref{eq:k_0}, if all of the inequalities
    \begin{align*}
        m_j &\le n-1 \text{ for all }j=1,\ldots,r+s+1,\\
        m_{r+2}+\cdots+m_{r+s+1}&\ge (s-1)(n-1),\\
        m_0+m_{r+2}+\cdots+m_{r+s+1}&\le s(n-1)
    \end{align*}
    hold. Otherwise, we have $\logTev^{X_{r,s,a}}_{\Gamma}=0$.
\end{theorem}

\begin{proof}
    As discussed earlier, by Proposition \ref{prop:X_rsa_all_transversality}, we have
    \begin{equation*}
    \logTev^{X_{r,s,a}}_{\Gamma}=\int_{Q_\Gamma(X_{r,s,a})}\prod_{i=1}^{n}[V(p_i,x_i)]
    \end{equation*}
    if the inequalities in the first two lines hold. By the calculation above, if the third inequality (which is equivalent to $k_0\ge m_0$) holds, then the integral on the right hand side equals the claimed expression, and equals zero otherwise. Finally, by Proposition \ref{prop:tev_nonzero_constraints}, we have $\logTev^{X_{r,s,a}}_{\Gamma}=0$ if the inequalities in the first two lines do not both hold.
\end{proof}

Note that the last two inequalities appearing in Theorem \ref{thm:X_rsa} imply that $m_0\le n-1$ is also necessary for $\logTev^{X_{r,s,a}}_{\Gamma}\neq 0$. The necessity of the inequality $k_0\ge m_0$ was a consequence of our integral calculation, but can also be deduced directly, along the lines of the proof of Proposition \ref{prop:tev_nonzero_constraints}. The same is true of the condition $m_0\le n-1$.

If $a>0$, then the given inequalities on $m_j$ are also sufficient for $\logTev^{X_{r,s,a}}_{\Gamma}\neq0$. Indeed, the last inequality is simply that $k_0\ge m_0$, which implies that 
\begin{equation*}
\left(\prod_{j=0}^{r+s+1}m_j!\right)\left(\prod_{j=0}^{r+s+1}\prod_{v=1}^{m_j}\mu_{j,v}\right)a^{k_0-m_0}\binom{k_0}{m_0}\neq 0.
\end{equation*}

\subsection{Specializations}\label{sec:specializations}

Theorem \ref{thm:X_rsa} recovers \cite[Theorem 7]{cil} in the case of Hirzebruch surfaces, $r=s=1$. Our $m_0,m_1,m_2,m_3$ are the $|\mu_2|,|\mu_1|,|\mu_3|,|\mu_4|$, respectively, appearing in \cite[Theorem 7]{cil}. If $s=1$ and $r\ge1$ is arbitrary, then Theorem \ref{thm:X_rsa} proves \cite[Conjecture 14]{cil}, and also makes precise exactly when the logarithmic Tevelev degree is zero. Taking $s$ arbitrary gives a more general statement.

The variety $X_{r,s,a}$ is Fano if and only if $a\le s$. When in addition $a\ge 2$, the fixed-domain Gromov-Witten invariants of $X_{r,s,a}$ are shown not to be enumerative in \cite[\S 3.2]{bllrst}, even in genus 0 for curve classes of large anti-canonical degree, disproving the main conjecture of \cite{lp}. In contrast, the corresponding logarithmic invariants $\logTev^{X_{r,s,a}}_{\Gamma}$, with $\mu_j=(1)^{m_j}$, are enumerative in genus 0, and are determined by Theorem \ref{thm:X_rsa}. The logarithmic setting is broadly understood to be more likely to yield enumerative counts of curves.

When $a=1$, we have $X_{r,s,a}\cong\Bl_{\PP^{s-1}}\PP^{r+s}$. Taking all of the $m_{j,v}$ equal to 1, substituting $c=d$ and $b=d-k$, and dividing by the factors of $m_j!$ that count orderings of the intersection points with the boundary gives that
\begin{equation*}
    \Tev^{\Bl_{\PP^{s-1}}\PP^{r+s}}_{0,n,\beta}=\binom{s(n-d-1)}{k}.
\end{equation*}
When $s=1$, this recovers the $\ell=1$ case of \cite[Theorem 12]{cl2} for the class $\beta$ of degree $d$ curves (against the hyperplane class) meeting the exceptional divisor with degree $k$. Without the logarithmic conditions, the $s>1$ case is new, 
but should also be no more difficult (even in higher genus, with some additional positivity assumptions on $\beta$) using the same methods of \cite{cl2}.

Taking $a=1$ and $b=c$, so that the curves in question do not intersect the exceptional divisor $D_0\subset \Bl_{\PP^{s-1}}\PP^{r+s}$, gives
\begin{equation*}
\logTev^{\PP^{r+s}}_\Gamma=\left(\prod_{j=1}^{r+s+1}m_j!\right)\left(\prod_{j=1}^{r+s+1}\prod_{v=1}^{m_j}\mu_{j,v}\right),
\end{equation*}
or, reindexing, simply
\begin{equation*}
\logTev^{\PP^{r}}_\Gamma=\left(\prod_{j=1}^{r+1}m_j!\right)\left(\prod_{j=1}^{r+1}\prod_{v=1}^{m_j}\mu_{j,v}\right).
\end{equation*}
Without tangency conditions, this specializes simply to $\Tev^{\PP^r}_{0,n,d}=1$ \cite[Theorem 1.1]{fl}.

Finally, when $a=0$, we have $X_{r,s,a}\cong\PP^r\times \PP^s$. Due to the presence of the factor $a^{k_0-m_0}$, we have
\begin{equation*}
\logTev^{\PP^{r}\times\PP^s}_\Gamma=\left(\prod_{j=0}^{r+s+1}m_j!\right)\left(\prod_{j=0}^{r+s+1}\prod_{v=1}^{m_j}\mu_{j,v}\right)
\end{equation*}
when
\begin{align*}
    m_0+m_{r+2}+\cdots+m_{r+s+1}&=s(n-1)\text{ and }\\
    m_1+\cdots+m_{r+1}&=r(n-1),
\end{align*}
and $\logTev^{\PP^{r}\times\PP^s}_\Gamma=0$ otherwise. 

Indeed, a log map $f:\PP^1\to\PP^r\times\PP^s$ is equivalent to the data of log maps to the factors, and in order for $f$ to pass through the expected number of points, it must be the case that exactly the correct number of conditions are imposed on both of the projections $\PP^1\to\PP^r$ and $\PP^1\to \PP^s$. In this case, $\logTev^{\PP^{r}\times\PP^s}_\Gamma$ is equal simply to the product of logarithmic Tevelev degrees for $\PP^r$ and $\PP^s$.

\section{Blow-ups of $\PP^2$}\label{sec:bl}

The value of $\logTev^X_{\Gamma}$ when $X$ is the blow-up of $\PP^r$ at $r$ points has also been predicted in \cite[Conjecture 15]{cil}. In this section, we show that the prediction does not always hold. For notational simplicity, we take $r=2$, though the arguments hold more generally.

We show more precisely that, when $X$ is the blow-up of $\PP^2$ at 2 points, $\logTev^X_{\Gamma}$ equals the predicted value in a particular range, by exhibiting the conjectured count as an intersection number on a moduli space of naive log quasimaps. However, this intersection number is not always enumerative.

This is similar to the situation of $\logTev^{X_{r,s,a}}_{\Gamma}$, where we were able to show that, in the remaining cases, the logarithmic Tevelev degree vanishes. In the setting of blow-ups of $\PP^r$, the range in which $\logTev^X_{\Gamma}=0$ does \emph{not} cover the remaining cases. We give an example in Theorem \ref{thm:conj15_false} where $\logTev^{X_{r,s,a}}_{\Gamma}$ can be proven neither to be 0 nor the value conjectured in \cite[Conjecture 15]{cil}.

\subsection{Naive log quasimaps}

Let $X=\Bl_{[0:1:0],[0:0:1]}(\PP^2)$. A fan $\Sigma$ for $X$ is given below.

\begin{figure}[H]
    \begin{center}
    \begin{tikzpicture}[xscale=0.55,yscale=0.55]
    %\draw [help lines] (-10,-10) grid (10, 10);

    \draw [thick,->] (0,0) -- (0,-2);
    \draw [thick,->] (0,0) -- (-2,0);
    \draw [thick,->] (0,0) -- (-2,-2);
    \draw [thick,->] (0,0) -- (2,0);
    \draw [thick,->] (0,0) -- (0,2);

    \node at (0,-2.5) {$\rho_1$};
    \node at (-2.5,0) {$\rho_2$};
    \node at (-2.5,-2.5) {$\rho_3$};
    \node at (2.5,0) {$\rho_4$};
    \node at (0,2.5) {$\rho_5$};

    \end{tikzpicture}
    \end{center}
\caption{Fan of $X=\Bl_{[0:1:0],[0:0:1]}(\PP^2)$}\label{fig:fan}
    
\end{figure}

A point of $X$ is given in coordinates by
\begin{equation*}
    x=[y_1y_2y_3:y_1y_4:y_2y_5].
\end{equation*}
The torus-invariant divisors $D_j=\text{div}(y_j)$ given by the vanishing of $y_j$ correspond to the rays $\rho_j$ of the fan $\Sigma$ in Figure \ref{fig:fan}. We identify the set $\Sigma(1)$ with $\{1,2,3,4,5\}$. The pairs of coordinates $\{y_1,y_2\},\{y_2,y_4\},\{y_3,y_4\},\{y_1,y_5\},\{y_3,y_5\}$ are not allowed to vanish simultaneously. The coordinates are taken up to simultaneous scaling of the collections of coordinates $\{y_3,y_4,y_5\},\{y_1,y_5\},\{y_2,y_4\}$. The 2-dimensional torus acts by scaling on $y_4$ and $y_5$. 

Fix tangency data $\Gamma$ for log maps $f:\PP^1\to X$. Let $d=\deg(f^{*}\cO_{\PP^2}(1))$. We construct a moduli space of naive log quasimaps as follows. Let $\cB$ be the product of copies $\PP^1_{j,v}\cong \PP^1$ parametrizing points of the divisors $D_j$. As in \S\ref{sec:quasimaps}, let $\nu:\PP^1\times\cB\to\cB$ be the projection, and define the universal divisors $\cD_{j,v},\cD_j\subset \PP^1\times\cB$. Define the moduli space of naive log quasimaps
 \begin{center}
     \begin{tikzcd}
     Q_\Gamma(X):=\PP\left(\nu_{*}(\cO_{\PP^1}(d)(-\cD_1-\cD_2-\cD_3))\oplus \nu_{*}(\cO_{\PP^1}(d)(-\cD_1-\cD_4)) \oplus \nu_{*}(\cO_{\PP^1}(d)(-\cD_2-\cD_5)) \right)  
     \arrow[d, "\pi"]
     \\
     \cB
     \end{tikzcd}
 \end{center}    

 $Q_\Gamma(X)$ is smooth and projective of dimension $m+2$. A point $f\in Q_\Gamma(X)$ is given by the data of divisors $D_j$ and three sections
\begin{align*}
    g_3&\in H^0(\PP^1,\cO(d)(-D_1-D_2-D_3)),\\
    g_4&\in H^0(\PP^1,\cO(d)(-D_1-D_4)),\\
    g_5&\in H^0(\PP^1,\cO(d)(-D_2-D_5)),
\end{align*}
taken up to scaling, and not all zero. 

A general point of $Q_\Gamma(X)$ gives rise to a log map $f:\PP^1\to X$. Indeed, suppose that none of $g_3,g_4,g_5$ is identically zero and that none of the underlying divisors $D_j$ (where $j=1,\ldots,5$) share common points in their supports. For $j=1,2$, let $1_{D_j}$ be a non-zero element of $H^0(\PP^1,\cO(D_j))$. Then,
\begin{equation}\label{eq:map_to_Bl}
    f=[1_{D_1}1_{D_2}g_3:1_{D_1}g_4:1_{D_2}g_5]
\end{equation}
is a point of $\cM_\Gamma(X)$. Moreover, the sections $1_{D_j}$ can be chosen \emph{consistently} over $\cB$, by restricting a choice of global section of $1_{\cD_j}$. In this way, $\cM_\Gamma(X)$ is identified with an open subset of $Q_\Gamma(X)$.

Let $Q^{\neq 0}_\Gamma(X)\subset Q_\Gamma(X)$ be the open subset where none of $g_3,g_4,g_5$ is identically zero. Let $\cM^{\bpf}_\Gamma(X)\subset Q^{\neq 0}_\Gamma(X)$ be the intermediate open set where the pairs of divisors 
\begin{equation*}
    \{D_1,D_2\},\{D_2,D_4\},\{D_3,D_4\},\{D_1,D_5\},\{D_3,D_5\}
\end{equation*}
share no common points in their supports. A point $g=(g_3,g_4,g_5)\in \cM^{\bpf}_\Gamma(X)$ also gives rise to a map $f:\PP^1\to X$ via \eqref{eq:map_to_Bl}. We have a stratification by open subsets
\begin{equation*}
    \cM_\Gamma(X)\subset\cM^{\bpf}_\Gamma(X)\subset Q^{\neq 0}_\Gamma(X)\subset Q_\Gamma(X)
\end{equation*}
as in the case of $X_{r,s,a}$.

Our space of naive log quasimaps $Q_\Gamma(X)$ is in some sense even more naive than the space $Q_\Gamma(X_{r,s,a})$ constructed in \S\ref{sec:quasimaps}. This is because $(g_3,g_4,g_5)$ in general do not even give rise to a quasimap in the sense of \cite{cfk}, as we allow $(g_3,g_4)=(0,0)$ or $(g_3,g_5)=(0,0)$ as long as not all of $g_3,g_4,g_5$ vanish. However, this moduli space will suffice for our purposes.

Let $X^\circ\subset X$ be the interior. Let $p\in\PP^1,x\in X^\circ$ be points, and write
\begin{equation*}
    x=[1\cdot1\cdot y_3:1\cdot y_4:1\cdot y_5].
\end{equation*}
Define the incidence locus $V(p,x)\subset Q_\Gamma(X)$ by two conditions:
\begin{enumerate}
    \item $y_4\cdot g_3-y_3\cdot g_4\in H^0(\PP^1,\cO(d)(-D_1))$ vanishes at $p$,
    \item $y_5\cdot g_3-y_3\cdot g_5\in H^0(\PP^1,\cO(d)(-D_2))$ vanishes at $p$.
\end{enumerate}
$V(p,x)$ restricts on $\cM^{\bpf}_\Gamma(X)$ to the locus where $f(p)=x$. 

Write $\cO_Q(1)$ for the dual of the universal sub-line bundle on $Q_\Gamma(X)$, and $\zeta$ for its first Chern class. Write also $\HH_{j,v}$ for the hyperplane class on the factor $\PP^1_{j,v}$ of $\cB$, and $\HH_j=\sum_{v=1}^{m_j}\HH_{j,v}$. Then, $V(p,x)$ is the common degeneracy locus of maps of line bundles
\begin{align*}
    \cO_Q(-1) &\to \nu_{*}(\cO(d)(-\cD_1)|_p),\\
    \cO_Q(-1) &\to \nu_{*}(\cO(d)(-\cD_2)|_p),
\end{align*}
so we have
\begin{equation*}
    [V(p,x)]=(\zeta-\HH_1)(\zeta-\HH_2).
\end{equation*}

\subsection{Logarithmic Tevelev degrees of $\Bl_{[0:1:0],[0:0:1]}(\PP^2)$}

Suppose that $n=\frac{m+2}{2}\ge3$ is an integer. Let $p_1,\ldots,p_n\in\PP^1$ and $x_1,\ldots,x_n\in X^\circ$ be general points. The degree of the subscheme
\begin{equation*}
    V:=\bigcap_{i=1}^{n}V(p_i,x_i)\subset Q_\Gamma(X)
\end{equation*}
is again a candidate for the value of $\logTev^X_\Gamma$. By adapting the arguments of the parallel statements for $X_{r,s,a}$, one may readily verify the following.

\begin{proposition}\label{prop:Bl(P^2)_statements}
\quad
    \begin{enumerate}
    \item[(i)] (cf. Proposition \ref{prop:transversality_open_locus}) $V$ is reduced of dimension 0 upon restriction to $\cM_\Gamma(X)$.
    \item[(ii)] (cf. Proposition \ref{prop:transversality_open_locus}) $V$ is empty upon restriction to $\cM^{\bpf}_\Gamma(X)-\cM_\Gamma(X)$.
    \item[(iii)] (cf. Proposition \ref{prop:bp_transversality}) $V$ is empty upon restriction to $Q^{\neq 0}_\Gamma(X)-\cM^{\bpf}_\Gamma(X)$.
    \item[(iv)] (cf. Proposition \ref{prop:sections cannot be zero}) Suppose the inequalities
    \begin{itemize}
        \item $m_4\le n-1$, $m_5\le n-1$,
        \item $m_1+m_3\le n-1, m_2+m_3\le n-1$
    \end{itemize}
    all hold. Then, $V$ is empty upon restriction to $Q_\Gamma(X)-Q^{\neq 0}_\Gamma(X)$.
    \item[(v)] (cf. Proposition \ref{prop:tev_nonzero_constraints}) If $\logTev^X_{\Gamma}\neq0$, then
    \begin{itemize}
        \item $m_3\le n-1$,
        \item $m_1+m_5\le n-1$, and
        \item $m_2+m_4\le n-1$.
    \end{itemize}
    \item[(vi)] (cf. \S\ref{sec:computation}) We have
    \begin{equation*}
        \int_{Q_\Gamma(X)}\prod_{i=1}^{n}[V(p_i,x_i)]=\left(\prod_{j=1}^{5}m_j!\right)\left(\prod_{j=1}^{5}\prod_{v=1}^{m_j}\mu_{j,v}\right)\binom{n-1-m_5}{m_1}\binom{n-1-m_4}{m_2}
    \end{equation*}
    if $m_1+m_5,m_2+m_4\le n-1$. Otherwise, the integral equals zero.
\end{enumerate}
\end{proposition}

\cite[Conjecture 15]{cil} predicts that $\logTev^X_\Gamma$ either equals 0 or the value of the integral in Proposition \ref{prop:Bl(P^2)_statements}(vi). Statements (i)-(iv) and (vi) of Proposition \ref{prop:Bl(P^2)_statements} show that $\logTev^X_\Gamma$ does in fact equal the right hand side of (vi) if
\begin{equation*}
    \max(m_1+m_3,m_2+m_3,m_4,m_5)\le n-1.
\end{equation*}
Likewise, statement (v) shows that $\logTev^X_{\Gamma}=0$ if
\begin{equation*}
    \max(m_1+m_5,m_2+m_4,m_3)\ge n.
\end{equation*}

However, this does not cover all possible cases. As a concrete example, take:
\begin{align*}
\mu_1=\mu_2&=(1),\\
\mu_3&=(1,1,1,1),\\
\mu_4=\mu_5&=(5).
\end{align*}
We have $(m_1,m_2,m_3,m_4,m_5)=(1,1,4,1,1)$, so $m=8$, $n=5$, and $d=6$.

\begin{theorem}\label{thm:conj15_false}
    Take $\Gamma$ defined by the partitions $\mu_j$ above. Then, we have
    \begin{equation*}
    \logTev^X_\Gamma=24\cdot 5\cdot 5\cdot 4.
    \end{equation*}
    This value is non-zero and differs from the value $24\cdot (5\cdot 5)\cdot \binom{3}{1}\cdot \binom{3}{1}$ of the integral in Proposition \ref{prop:Bl(P^2)_statements}(vi). In particular, \cite[Conjecture 15]{cil} does not hold.
\end{theorem}

\begin{proof}
    By Proposition \ref{prop:Bl(P^2)_statements}(i)-(iii), the intersection
    \begin{equation*}
        V=\bigcap_{i=1}^{5}V(p_i,x_i)\subset Q_\Gamma(X)
    \end{equation*}
    is the disjoint union of $\logTev^X_\Gamma$ many reduced points in $\cM_\Gamma(X)$, and a closed subscheme $Z\subset Q_\Gamma(X)-Q^{\neq0}_\Gamma(X)$.

    A straightforward analysis as in the proof of Proposition \ref{prop:sections cannot be zero} shows that, if $f\in V\cap (Q_\Gamma(X)-Q^{\neq0}_\Gamma(X))$, then we must have $g_4=g_5=0$, and the $g_3$ must vanish at all of $p_1,\ldots,p_5$ when viewed either as a section of $\cO(6)(-D_1)=\cO(6)(-q_{1,1})$ or $\cO(6)(-D_2)=\cO(6)(-q_{2,1})$. Concretely, this means that
    \begin{equation*}
    D_3+q_{1,1}=D_3+q_{2,1}=p_1+p_2+p_3+p_4+p_5.
    \end{equation*}
    Therefore, we need $q_{1,1}=q_{2,1}=p_i$ for some $i=1,\ldots,5$, and $D_3$ is determined uniquely, although we have $4!=24$ ways to order the points of $D_3$. The remaining divisors $D_4=5q_{4,1}$ and $D_5=5q_{5,1}$ are free to vary in the factors $\PP^1_{4,1}\times \PP^1_{5,1}$ of $\cB$. The section $g_3\in H^0(\PP^1,\cO(6)(-D_3))$ is determined up to scaling.

    Thus, $Z$ consists of $24\cdot 5$ components, each isomorphic to $\PP^1_{4,1}\times \PP^1_{5,1}$; in Lemma \ref{lem:excess_locus} below, we show that this is true as a \emph{scheme}. More precisely, each copy of $\PP^1_{4,1}\times \PP^1_{5,1}$ is viewed as a subscheme of $Q_\Gamma(X)$ via the section $s:\cB\hookrightarrow  Q_{\Gamma}(X)$ cut out by the equations $g_4=g_5=0$.

    Write $Z_\bullet\subset Z$ for any one of these components. By the excess intersection formula, each component $Z_\bullet$ contributes
    \begin{equation*}
        \int_{\PP^1_{4,1}\times \PP^1_{5,1}}\frac{((1+\zeta-\HH_1)(1+\zeta-\HH_2))^5}{c(N_{Z_{\bullet}/Q_\Gamma(X)})}
    \end{equation*}
    to the integral 
    \begin{equation*}
        \int_{Q_\Gamma(X)}\prod_{i=1}^{5}[V(p_i,x_i)]=24\cdot (5\cdot 5)\cdot \binom{3}{1}\cdot \binom{3}{1}.
    \end{equation*}

    The inclusion $Z_\bullet\hookrightarrow Q_{\Gamma}(X)$ factors as
    \begin{equation*}
        Z_\bullet \hookrightarrow \cB \hookrightarrow Q_{\Gamma}(X),
    \end{equation*}
    where $Z_\bullet\hookrightarrow \cB$ is the inclusion determined by the given ordering of the points $p_i$, and $\cB \hookrightarrow Q_{\Gamma}(X)$ is the section $s$ cut out by the equations $g_4=g_5=0$. The rank 6 sub-bundle $N_{Z_\bullet /\cB}\subset N_{Z_\bullet/Q_{\Gamma}(X)}$ is trivial, because it is given by an inclusion $(\PP^1)^2\subset (\PP^1)^8$, which is pulled back from an inclusion of a point in $(\PP^1)^6$. The quotient is cut out by the vanishing of $g_4\in H^0(\PP^1,\cO(6)(-D_1-D_4))$ and $g_5\in H^0(\PP^1,\cO(6)(-D_2-D_5))$. Therefore, 
        \begin{equation*}
        \int_{\PP^1_{4,1}\times \PP^1_{5,1}}\frac{((1+\zeta-\HH_1)(1+\zeta-\HH_2))^5}{c(N_{Z_{\bullet}/Q_\Gamma(X)})}=\int_{\PP^1_{4,1}\times \PP^1_{5,1}}\frac{((1+\zeta-\HH_1)(1+\zeta-\HH_2))^5}{(1+\zeta-\HH_1-\HH_4)(1+\zeta-\HH_2-\HH_5)}.
    \end{equation*}

    The classes $\HH_1,\HH_2,\HH_3$ clearly all restrict to zero on $\PP^1_{4,1}\times \PP^1_{5,1}$. The same is true of the class $\zeta$. Indeed, the restriction of the universal sub-bundle
    \begin{equation*}
        \cO_{Q}(-1)\to \nu_{*}(\cO(6)(-\cD_1-\cD_2-\cD_3)) \oplus \nu_{*}(\cO(6)(-\cD_1-\cD_4))\oplus \nu_{*}(\cO(6)(-\cD_2-\cD_5))
    \end{equation*}
    to the locus where $g_4=g_5=0$ is simply the inclusion of the first factor
        \begin{equation*}
        \nu_{*}(\cO(6)(-\cD_1-\cD_2-\cD_3))\to \nu_{*}(\cO(6)(-\cD_1-\cD_2-\cD_3)) \oplus \nu_{*}(\cO(6)(-\cD_1-\cD_4))\oplus \nu_{*}(\cO(6)(-\cD_2-\cD_5)).
    \end{equation*}
    However, the first factor becomes trivial upon restriction to $Z_{\bullet}$, so $\cO_Q(-1)$ (and hence $\cO_Q(1)$) is trivial upon restriction to $Z_\bullet$. 
    
    Thus, the integral becomes
           \begin{equation*}
\int_{\PP^1_{4,1}\times \PP^1_{5,1}}\frac{1}{(1-\HH_4)(1-\HH_5)}=\int_{\PP^1_{4,1}\times \PP^1_{5,1}}(1+5\HH_{4,1})(1+5\HH_{5,1})=5\cdot 5.
    \end{equation*}

    Combining the contributions to $\int_{Q_\Gamma(X)}\prod_{i=1}^{5}[V(p_i,x_i)]$, we have
            \begin{equation*}
        (24\cdot5)\cdot(5\cdot5)+\logTev^{X}_\Gamma=24\cdot (5\cdot 5)\cdot \binom{3}{1}\cdot \binom{3}{1},
    \end{equation*}
    from which the conclusion follows.
\end{proof}

We used in the proof of Theorem \ref{thm:conj15_false}:

\begin{lemma}\label{lem:excess_locus}
The excess locus $Z_\bullet\subset V$ is isomorphic as a scheme to $\PP^1\times\PP^1$.
\end{lemma}

\begin{proof}
We will show that the only tangent vectors lying in $Z_\bullet$ are those lying along $\PP^1_{4,1}\times \PP^1_{5,1}$. Without loss of generality, suppose that $Z_\bullet$ is the component of $V$ where
\begin{align*}
D_1&=p_5,\\
D_2&=p_5,\\
D_3&=p_1+p_2+p_3+p_4,\\
D_4&=5q_{4},\\
D_5&=5q_{5}
\end{align*}
We write $q_4,q_5$ for the variable points instead of $q_{4,1},q_{5,1}$ for brevity. 

Let $z$ be a coordinate on $\PP^1$, and view the points $p_i,q_j\in \bC$ as complex numbers. The sections $g_3,g_4,g_5$ underlying points of $Z_\bullet$ are viewed as polynomials of degree 6 vanishing along the required divisors. Then, a tangent vector in $Z_\bullet$ may be written as
\begin{align*}
[&(1+\gamma_3\epsilon)(z-p_5-\gamma_{11}\epsilon)(z-p_5-\gamma_{21}\epsilon)(z-p_1-\gamma_{31}\epsilon)(z-p_2-\gamma_{32}\epsilon)(z-p_3-\gamma_{33}\epsilon)(z-p_4-\gamma_{34}\epsilon):\\
&\gamma_4\epsilon(z-p_5-\gamma_{11}\epsilon)(z-q_4-\gamma_{41}\epsilon)^5:\\
&\gamma_5\epsilon(z-p_5-\gamma_{11}\epsilon)(z-q_5-\gamma_{51}\epsilon)^5],
\end{align*}
where the coordinates are first-order deformations of $g_3$ and $g_4,g_5=0$. The $\gamma_{jv}\epsilon$ are deformations of the points $q_{j,v}$, and $\gamma_3\epsilon,\gamma_4\epsilon,\gamma_5\epsilon$ are deformations in the directions of the fibers of $\pi:Q_\Gamma(X)\to\cB$. This triple of polynomials must satisfy the equations defining $V(p_i,x_i)$ for $i=1,\ldots,5$, viewed in $\bC[\epsilon]/\epsilon^2$.

Projecting to the last two coordinates, using the fact that the points $p_i$ are distinct, and looking only at the first order terms, the pair of sections
\begin{equation*}
[\gamma_4(z-q_4)^5:\gamma_5(z-q_5)^5]
\end{equation*}
must satisfy four general incidence conditions upon evaluating at $z=p_1,p_2,p_3,p_4$. An easy incidence correspondence argument (also taking into account the possibility $p_i=q_j$) shows that this is impossible unless $\gamma_4=\gamma_5=0$. Comparing now to the deformation of $g_3$ shows that all of the $\gamma_{jv}$ appearing in $g_3$ must be zero. Our tangent vector is now simply
\begin{equation*}
[(1+\gamma_3\epsilon)g_3:0:0],
\end{equation*}
where the points $q_4,q_5$ are free to vary along $\PP^1_{4,1}\times\PP^1_{5,1}$ (corresponding to the parameters $\gamma_{41},\gamma_{51}$). Up to scaling, the parameter $\gamma_3$ can be ignored. Thus, our tangent vector lies in $\PP^1_{4,1}\times\PP^1_{5,1}$, completing the proof.

\end{proof}

\bibliographystyle{alpha} 
\bibliography{logtev_v4.bib}

\end{document}